\documentclass[microtype]{gtpart} 
\usepackage{graphicx}  
\title{Partial duals of plane graphs, separability and the graphs of knots}

\author{Iain Moffatt}
\givenname{Iain}
\surname{Moffatt}
\address{Department of Mathematics and Statistics,  University of South Alabama, Mobile, AL 36688, USA.}
\email{imoffatt@jaguar1.usouthal.edu}
\urladdr{http://www.southalabama.edu/mathstat/personal${}_{-}$pages/imoffatt/}


\arxivreference{1007.4219}

\volumenumber{}
\issuenumber{}
\publicationyear{}
\papernumber{}
\startpage{}
\endpage{}
\doi{}
\MR{}
\Zbl{}
\received{}
\revised{}
\accepted{}
\published{}
\publishedonline{}
\proposed{}
\seconded{}
\corresponding{}
\editor{}
\version{}

\theoremstyle{definition}
\newtheorem{theorem}{Theorem}[section]
\newtheorem{proposition}[theorem]{Proposition}
\newtheorem{lemma}[theorem]{Lemma}
\newtheorem{corollary}[theorem]{Corollary}
\newtheorem{theorem*}{Main Theorem}

\theoremstyle{definition}
\newtheorem{definition}[theorem]{Definition}

\theoremstyle{remark}
\newtheorem{remark}[theorem]{Remark}
\newtheorem{example}[theorem]{Example}

\newcommand{\G}{\mathbb{G}}
\newcommand{\T}{T}
\newcommand{\calD}{\mathcal{D}}
\newcommand{\pb}{plane-biseparation }
\newcommand{\pbs}{plane-biseparations }
\newcommand{\pbt}{plane-biseparation}
\newcommand{\pbst}{plane-biseparations}
\newcommand{\jpb}{plane-join-biseparation }
\newcommand{\jpbs}{plane-join-biseparations }
\newcommand{\jpbt}{plane-join-biseparation}
\newcommand{\jpbst}{plane-join-biseparations}

\newcommand{\bs}{\backslash}

\begin{document}

\begin{abstract}    
There is a well-known way to describe a link diagram as a (signed) plane graph, called its Tait graph. This concept was recently  extended, providing a way to  associate a set of embedded graphs (or ribbon graphs) to a link diagram. 
While every plane graph arises as a Tait graph of  a unique link diagram, not every embedded graph represents a link diagram. Furthermore, although a Tait graph describes a unique link diagram, the same embedded graph can represent many different link diagrams.  
One is then led to ask which embedded graphs represent link diagrams, and how link diagrams presented by the same embedded graphs are related to one another. Here we answer these questions by characterizing
  the class of embedded graphs that represent link diagrams, and then using this characterization to find a move  that relates all of the link diagrams that are presented by the same set of embedded graphs.    
\end{abstract}

\maketitle 


\section{Overview}
\subsection{Introduction and motivation}
There is a classical and well-known way to associate a (signed) plane graph, called a {\em Tait graph} (or {\em checkerboard graph} or {\em $2$-face graph}), to the diagram of a link: start with a checkerboard colouring of a link diagram (i.e. colour each face black or white in such a way that adjacent faces are of different colour); place a vertex in each black face; add an edge between vertices corresponding to incident black faces; and sign the edges to record the crossing type  (see Subsection~\ref{ss.tg} and Figure~\ref{f.taitgraph}). 
Tait graphs are  a standard tool in knot theory. They provide a bridge between knot theory and graph theory,  and have found numerous applications in both of these areas. 

Since there are exactly two possible checkerboard colourings of a link diagram, every link diagram admits exactly two (signed) Tait graphs. Moreover, these two graphs are geometric duals of each other (see  Subsection~\ref{ss.duals} for a definition of duality, noting that duality changes the sign of an edge.) Also, as a unique link diagram can be obtained from every signed plane graph,   every signed plane graph arises as a Tait graph (see Subsection~\ref{ss.med}). Thus Tait graphs  have the following properties:
\begin{enumerate}
\renewcommand{\theenumi}{T\arabic{enumi}}
\item \label{t1} The two Tait graphs associated with a link diagram are geometric duals. 
\item  \label{t2} Every signed plane graph is the Tait graph of a link diagram.
\item \label{t3} A Tait graph gives rise to a unique link diagram.
\end{enumerate}

Recently, in \cite{Detal}, Dasbach,  Futer,  Kalfagianni,  Lin and Stoltzfus (see also Turaev~\cite{Tu97}) extended the idea of a Tait graph  by associating a set of  embedded graphs to a link diagram. In this construction, each embedded graph arises by assigning one of the two possible smoothings at each crossing of the link (see Subsection~\ref{ss.rgl} and Figure~\ref{f.rgoflink}).   The Tait graphs of a link diagram appear in this set of embedded graphs. 
One of the key advantages to this approach of using non-plane graphs to describe links is that it provides a way to encode the crossing structure of a link diagram in the topology of the embedded graph, rather than by using signs on the edges. 
This idea has proved to be  very useful  and  has found many recent applications in knot theory, such as to the Jones and {\small HOMFLY-PT} polynomials \cite{CP,CV,Ch1, Detal, Mo2, Mo3, VT10}, Khovanov homology \cite{CKS07, Detal}, knot Floer homology \cite{Lo08},    Turaev genus  \cite{Ab09,Lo08,Tu97},
quasi-alternating links \cite{Wi09},  the coloured Jones polynomial \cite{FKP09}, the signature of a knot \cite{DL10}, the  determinant of a knot \cite{Detal, Detal2}, and hyperbolic knot theory \cite{FKP08}.

Given the breadth  of applications, understanding the structure of the set of embedded graphs of a knot  is a fundamental and important problem.   
By considering  the properties~\ref{t1}-\ref{t3} of Tait graphs, we are led to ask 
what the corresponding properties for the  more general embedded graphs of a link diagram are:
\begin{enumerate}
\renewcommand{\theenumi}{Q\arabic{enumi}}
\item \label{q1} How are the embedded graphs of a link diagram related to each other?
\item  \label{q2} Which embedded graphs arise as embedded graphs of a link diagram?
\item \label{q3} What is the relation between link diagrams that are presented by  the same signed embedded graphs?  
\end{enumerate}
Here we answer these three questions, and in doing so we introduce and develop a theory that relates the structure and the topology of an embedded graph.

\subsection{Overview of results}
The answer to Question~\ref{q1} is known (see \cite{Ch1}): just as Tait graphs are geometric duals, all of the embedded graphs of a link diagram are partial duals.  Partial duality is a recent extension of the concept of geometric duality. Loosely speaking, a partial dual of an embedded graph is obtained by forming its geometric dual  with respect to only a subset of its edges (see Subsection~\ref{ss.pd}).
Partial duality was introduced by Chmutov in \cite{Ch1} to relate various realizations of the Jones polynomial as a graph polynomial (see also \cite{Mo2}). It has since found several applications to knot theory, graph theory and physics (see \cite{Ch1,EMM,HM11,HMV,KRVT09,Mo2,Mo3, Mo4,VT,VT10}), and it appears to be a fundamental operation on an embedded graph.  

While every plane graph arises as a Tait graph of a link diagram (Property~\ref{t2}), it is {\em not} the case that every embedded graph represents a link diagram. Question \ref{q2} is therefore non-trivial. A solution to \ref{q2}, however, arises through connections with Tait graphs:
since Tait graphs belong to the  set of embedded graphs of a link diagram, it follows that an embedded graph represents a link diagram if and only if it is the partial dual of a plane graph. Thus \ref{q2} can be reformulated as the graph theoretical problem of characterizing partial duals of plane graphs. We characterize this class of embedded graphs in terms of separability. We say that an embedded graph $G$ is a $1$-sum of two non-trivial subgraphs $P$ and $Q$ if $G=P\cup Q$, and $P\cap Q$ is a single vertex (see Figure~\ref{f.join} and Subsection~\ref{ss.sum}).  In Section~\ref{s.1-sums} we introduce the concept of a \pbt, which, loosely speaking, says that an embedded graph can be constructed by $1$-summing plane graphs from two sets, $P_A$ and $Q_A$, in such a way that every $1$-sum involves exactly one graph in $P_A$ and exactly one  in $Q_A$. 
Our first main result (cf. Theorems~\ref{t.pchar1} and~\ref{t.pchar2}) characterizes the partial duals of plane graphs, and hence answers \ref{q2}: 
\begin{theorem*}
An embedded graph is the partial dual of a plane graph if and only if it admits a \pbt.
Consequently, an embedded graph represents a link diagram if and only if it admits a \pbt.
\end{theorem*}
This theorem also provides an interesting connection between the structure of an embedded graph and the genus of its partial duals. This connection is further developed in \cite{Mo6}.

Moving on to Question~\ref{q3}, we first note that while every Tait graph represents a unique link diagram (Property~\ref{t1}), this is not the case for embedded graphs: an embedded graph can represent many different link diagrams. Given an embedded graph $G$, to recover a link diagram it represents, one first has to obtain a plane partial dual, and then, from this, form a link diagram. An embedded graph can have many plane partial duals, and the fact that an embedded graph can represent many different link diagrams is a direct consequence of this. Thus to answer  \ref{q3}, we need to relate plane partial duals. To do this we introduce a simple move on embedded graphs called a dual of a join-summand move (see Definition~\ref{d.djs}). We then prove (cf. Theorem~\ref{t.sim}) that this simple move relates all plane partial duals:
\begin{theorem*}
Let $G$ and $H$ be plane graphs. Then $G$ and $H$ are partial duals if and only if they are related by dual of a join-summand moves.
\end{theorem*}
We then apply  this theorem in Section~\ref{s.knots} to  answer \ref{q3} by relating all of the link diagrams represented by the same embedded graphs. We define a summand-flip to be a move on a link diagram that ``flips over'' a connected sum component of a link diagram (see Figure~\ref{f.sf} and Definition~\ref{d.lflip}). We prove:
\begin{theorem*}
Link diagrams $D$ and $D'$ are represented by the same set ribbon graphs if and only if they are related by a sequence of summand-flips.
\end{theorem*}
This theorem appears as Theorem~\ref{t.diagrams}. In particular, this means that if $D$ and $D'$ are represented by the same embedded graphs, then they represent the same link diagram.

\medskip

This paper is structured as follows. In Section~\ref{s.emb} we review some properties of embedded graphs and their representations. In Section~\ref{s.gol} we recall the construction of Tait graphs and the embedded graphs of a link diagram. In Section~\ref{s.pd}, we describe partial duality, its relation to link diagrams, and use it to reformulate \ref{q2} and \ref{q3}. In Section~\ref{s.1-sums} we introduce \pbs and other related decompositions of ribbon graphs.
Section~\ref{s.char} contains the characterization of partial duals of plane graphs in terms of \pbst. In Section~\ref{s.join} we study and relate plane partial duals. Finally, in Section~\ref{s.knots}, we provide answers to \ref{q1}-\ref{q3}.

The graph theoretical results presented here  are of interest in their own right, and a reader interested in only these results may prefer to skip the knot theory, reading only Sections~\ref{s.emb}, \ref{ss.pd}, and \ref{s.1-sums}-\ref{s.join}.

\section{Embedded graphs}\label{s.emb}
In this section we review our basic objects of study:  embedded graphs, ribbon graphs, and arrow presentations. 
Our terminology is standard, and a reader familiar with ribbon graphs, their realizations as cellularly embedded graphs and arrow presentations, and with geometric duality may safely skip this section.

\subsection{Embedded graphs and their descriptions }

\begin{figure}
\begin{tabular}{ccccc}
\raisebox{-10mm}{\includegraphics[height=2cm]{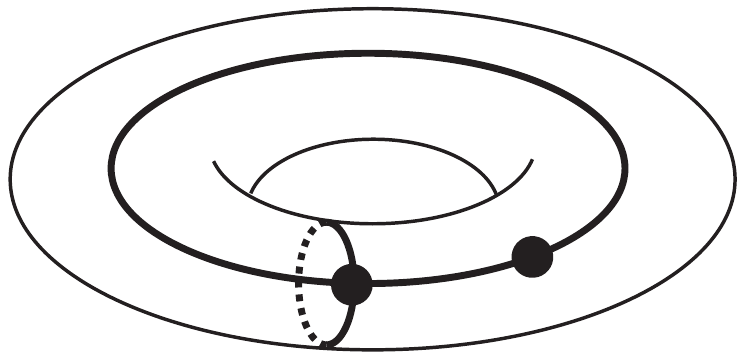}} &\raisebox{0mm}{=}   & \raisebox{-10mm}{\includegraphics[height=2cm]{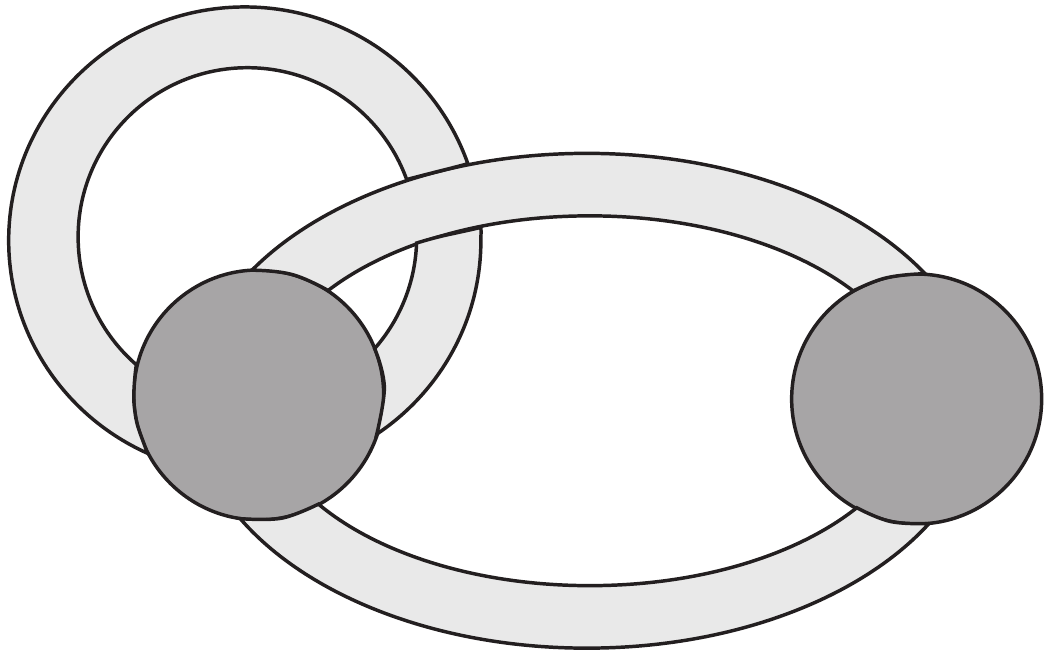}} &\raisebox{0mm}{=}   & \raisebox{-7mm}{\includegraphics[height=1.5cm]{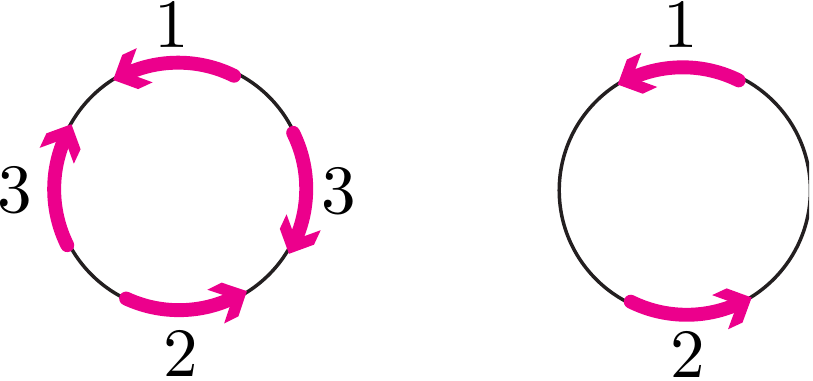} } \\
&&&&\\
Cellularly embedded graph. & &Ribbon Graph. & & Arrow presentation.
\end{tabular}
\caption{Three descriptions of the same embedded graph.}
\label{f.desc}
\end{figure}

\subsubsection{Cellularly embedded graphs}

An {\em embedded  graph} $G=(V(G),E(G)) \subset \Sigma$ is a graph drawn on surface $\Sigma$  in such a way that edges only intersect at their ends. The arcwise-connected components of $\Sigma \bs G$ are called the {\em regions} of $G$, and regions homeomorphic to discs are called {\em faces}. If each of the regions of an embedded graph $G$ is a face we say that $G$ is a {\em cellularly embedded graph}. An example of a graph cellularly embedded in the torus is given in Figure~\ref{f.desc}. Note that if a cellularly embedded graph $G\subset \Sigma$ is not  connected, then each component of $G$ is cellularly embedded in a distinct component of $\Sigma$.

 The {\em genus} a cellularly embedded graph is the genus of the surface it is embedded in. 
 A {\em plane graph} is a graph that is cellularly embedded in a genus zero surface, i.e. each of its components is cellularly embedded in a sphere $S^2$. 
 
Two embedded graphs, $G\subset \Sigma$ and  $G'\subset \Sigma'$  are  {\em equivalent}   if there is a homeomorphism from  $\Sigma$ 
to $\Sigma'$ that sends $G$ to $G'$. As is standard, we consider embedded graphs up to equivalence.

\subsubsection{Ribbon graphs}
One of the difficulties when working with a cellularly embedded graph $G$ is that deleting  one of its edges may result in a graph that is not cellularly embedded. More generally, if $G\subset \Sigma$ is cellularly embedded, then although a  subgraph $H$ of $G$ is embedded in $\Sigma$, it need not be cellularly embedded. 
In order to get around this difficulty we work in the language of ribbon graphs. Ribbon graphs correspond to cellularly embedded graphs, but have the advantage that subgraphs (and minors) of ribbon graphs are also ribbon graphs - a property that is vital here.

\begin{definition}
A {\em ribbon graph} $G =\left(  V(G),E(G)  \right)$ is a (possibly non-orientable) surface with boundary represented as the union of two  sets of topological discs, a set $V (G)$ of {\em vertices}, and a set of {\em edges} $E (G)$ such that: 
\begin{enumerate}
\item the vertices and edges intersect in disjoint line segments;
\item each such line segment lies on the boundary of precisely one
vertex and precisely one edge;
\item every edge contains exactly two such line segments.
\end{enumerate}
\end{definition}

Ribbon graphs are considered up to  homeomorphisms of the surface that preserve the vertex-edge structure.  

A ribbon graph is said to be {\em orientable} if it is orientable as a surface. Here we will only consider orientable ribbon graphs. The {\em genus}, $g(G)$, of a ribbon graph is the genus of $G$ as a punctured surface.  A ribbon graph is {\em plane} if it is of genus zero.  We emphasize that here a plane ribbon graph need not be connected.

Ribbon graphs are easily seen (and well-known) to be equivalent to cellularly embedded 
graphs. Intuitively, if $G$ is a cellularly embedded graph, a ribbon 
graph representation results from taking a small neighbourhood  of 
the cellularly embedded graph $G$. On the other hand, if $G$ is a 
ribbon graph, we simply sew discs into each boundary component of the 
ribbon graph to get the desired surface. See Figure~\ref{f.desc}.

A ribbon graph $H =\left(  V(H),E(H)  \right)$ is a {\em ribbon subgraph} of $G =\left(  V(G),E(G)  \right)$ if $H$ can be obtained by deleting vertices and edges of $G$. If $V(H)=V(G)$, then the ribbon subgraph $H$ is a  {\em spanning ribbon subgraph} of $G$.  If $A\subseteq E(G)$, then the ribbon subgraph {\em induced} by $A$, denoted $G|_A$, is the ribbon subgraph of $G$ that consists of the edges in $A$ and their incident vertices. 
We will often regard a ribbon subgraph $H$ as being embedded in $G$, and will often identify the vertices and edges of $H$ with the corresponding vertices and edges in $G$. 

Throughout the paper we use $A^c:=E(G)\bs A$  to denote the complement of $A\subseteq E(G)$.

\subsubsection{Arrow presentations}

At times we will find it particularly convenient to represent ribbon graphs as arrow presentations.  
 \begin{definition}[Chmutov \cite{Ch1}] An {\em arrow presentation} consists of  a set of  circles,  each with a collection of disjoint,  labelled arrows, called {\em marking arrows}, lying on them. Each label appears on exactly two arrows. 
\end{definition}

An arrow presentation is shown in Figure~\ref{f.desc}.

 Two arrow presentations are considered equivalent if one can be obtained from the other by reversing the direction of all of the marking arrows which belong to some subset of labels, or by changing the labelling set. The circles in an arrow presentation are considered up to homeomorphism.   We consider all arrow presentations up to equivalence.

A ribbon graph can be obtained from an arrow presentation as follows. View each circle as the boundary of a disc. Each disc becomes a vertex of the ribbon graph.  Edges are then added to the vertex discs in the following way. Take an oriented disc for each label of the marking arrows.  Choose two non-intersecting arcs on the boundary of each of the edge discs and direct these according to the orientation. Identify these two arcs with two marking arrows, both with the same label, aligning the direction of each arc consistently with the orientation of the marking arrow. This process can be illustrated pictorially thus:
\[ \includegraphics[height=10mm]{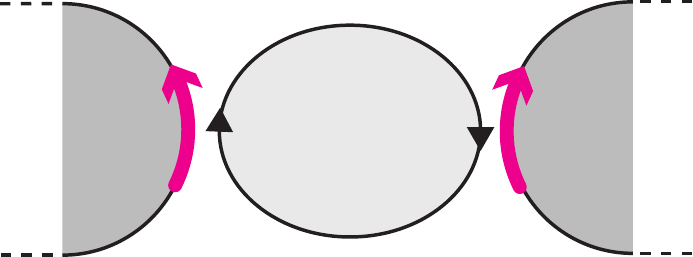}  \hspace{4mm}
\raisebox{4mm}{\includegraphics[width=11mm]{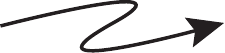}}\hspace{4mm}
\includegraphics[height=10mm]{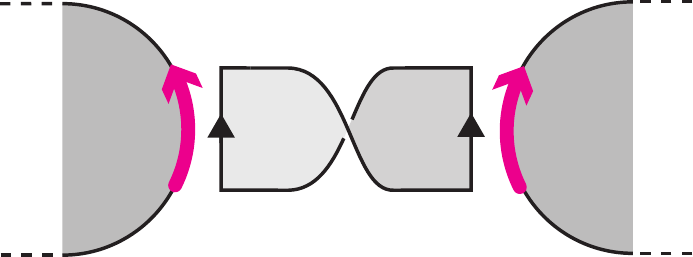} \hspace{4mm}
\raisebox{4mm}{\includegraphics[width=11mm]{arrow}}\hspace{4mm}
 \includegraphics[height=10mm]{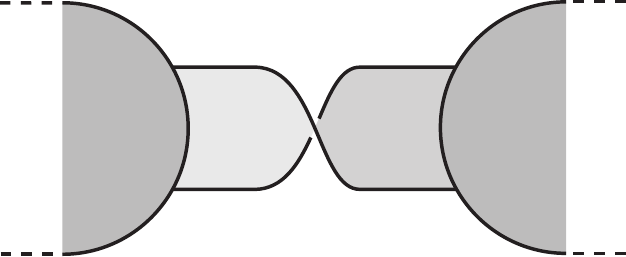} \; . \]

Conversely, to describe a ribbon graph $G$ as an arrow presentation, start by arbitrarily labelling and orienting the boundary of each edge disc of $G$.  On the arc where an edge disc intersects a vertex disc, place an arrow on the vertex disc, labelling the arrow with the label of the edge it meets and directing it consistently with the orientation of the edge disc boundary. The boundaries of the vertex set marked with these labelled arrows give the arrow marked circles of an arrow presentation.  See Figure~\ref{ribbon arrow} for an example, and \cite{Ch1} for further details.

\begin{figure}
\[ \raisebox{35mm}{\rotatebox{-90}{\includegraphics[width=35mm]{f1a}}}\quad
\raisebox{14mm}{\includegraphics[width=15mm]{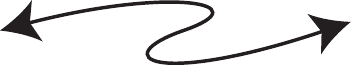}} \quad
 \includegraphics[height=35mm]{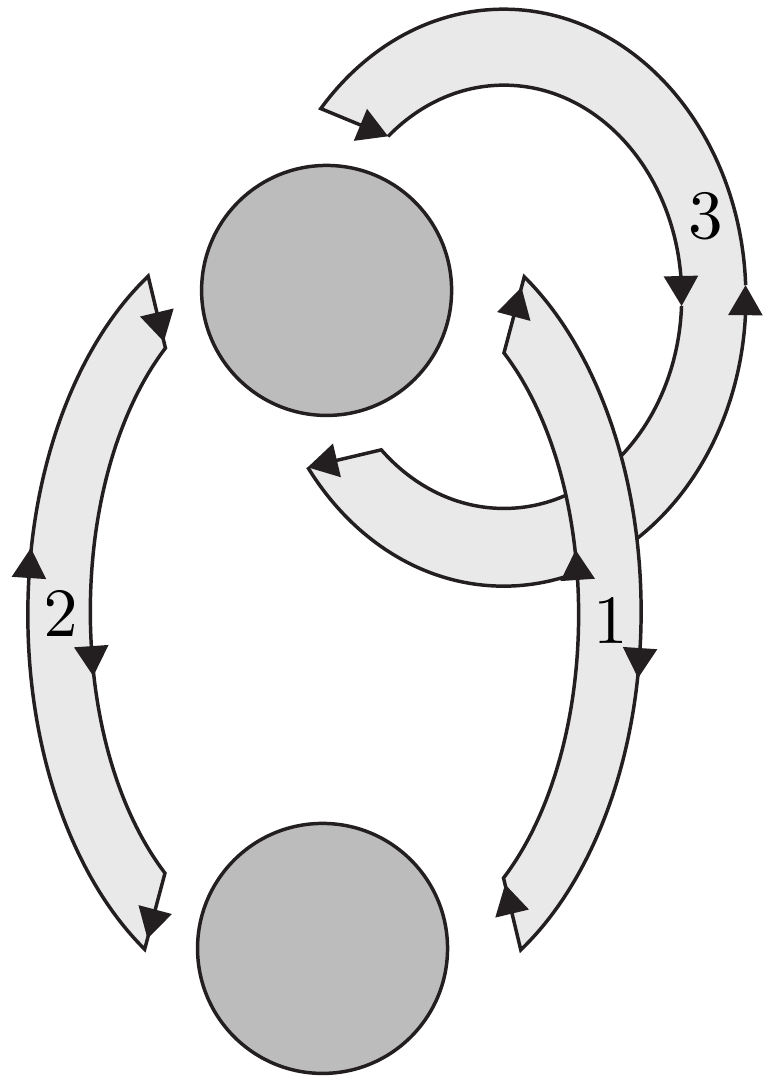} \quad
 \raisebox{14mm}{\includegraphics[width=15mm]{doublearrow}} \quad
\raisebox{1mm}{ \includegraphics[height=35mm]{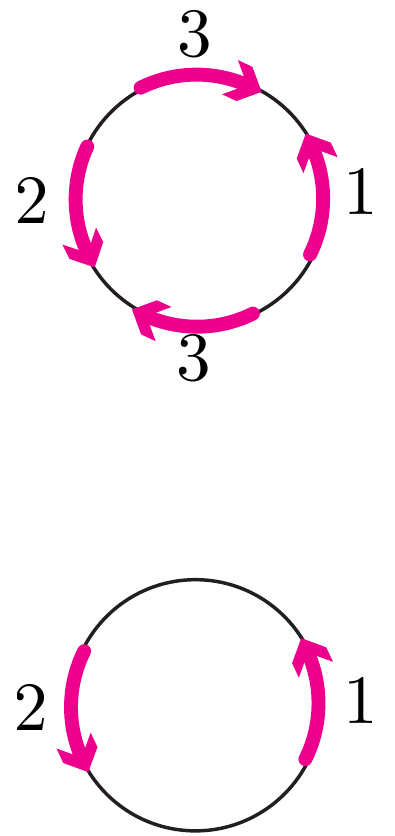}}\]
\caption{Equivalence of arrow presentations and ribbon graphs.}
\label{ribbon arrow}
\end{figure}

\subsection{Signed embedded graphs}
A ribbon graph or a cellularly embedded graph $G$ is said to be {\em signed} if it is equipped with a   mapping from its edge set $E(G)$ to $ \{+,-\}$ (so a sign ``$+$'' or ``$-$'' is assigned to each edge of $G$). An arrow presentation is said to be {\em signed} if there is a mapping  from the set of labels of the arrows to $\{+,-\}$. The equivalence between cellularly embedded graphs, ribbon graphs and arrow presentations clearly extends to their signed counterparts.

We will often want to consider embedded graphs that are obtained from signed embedded graphs by forgetting the signs. We will refer to such embedded graphs as being {\em unsigned}. At times we will also use the term ``unsigned'' to emphasize that a ribbon graph is not equipped with signs. This double use of the term should cause no confusion.

\subsection{Geometric duals}\label{ss.duals}

The construction of the {\em geometric dual} $G^*\subset \Sigma$ of a cellularly embedded 
graph, $G\subset \Sigma$, is well known: form $G^*$ by placing one 
vertex in each face of $G$ and embed an edge of $G^*$ between two 
vertices whenever the faces of $G$ they lie in are adjacent. 

If $G$ has $k$  components, $G_1,\ldots, G_k$,  and is cellularly embedded in a surface, then each component of the graph is cellularly embedded in a connected component of the surface, and therefore duality acts disjointly on  components of the graph:  $(G)^* = G_1^* \sqcup \cdots \sqcup G_k^*$.

There is a natural bijection between the edges of $G$ and the edges of $G^*$. We will use this bijection to identify the edges of $G$ and the edges of $G^*$.

Geometric duals have a particularly neat description in the language of ribbon graphs. If $G$ is a ribbon graph, then its geometric dual $G^*$ is obtained from $G$ in the following way: regard $G=(V(G), E(G))$ as a punctured surface. By filling in the punctures using a set of discs denoted $ V(G^*) $, we obtain a surface without boundary  $\Sigma$.  $G^*$  is then the ribbon graph with vertex set  $V(G^*)$  and edge set $E(G) $  obtained as $\Sigma \bs V(G)$. Observe that any decorations on the boundary of $G$ give decorations on the  boundary of $G^*$. We will use this fact later.

By definition, duality changes the sign of an edge in a signed ribbon graph. That is, if $G$ is a signed ribbon graph or signed embedded graph, with $A\subseteq E(G)$ its set of positive edges, and $A^c$ its set of negative edges, then $G^*$ is also signed, but $A\subseteq E(G^*)$ is its set of negative edges and $A^c\subseteq E(G^*)$ is its set of positive edges.

\section{The graphs of a link diagram}\label{s.gol}
 In this section we start by reviewing the construction of a Tait graph and its basic properties. 
We then describe the extension of Tait graphs to the ribbon graphs of a link diagram due to Turaev \cite{Tu97} and Dasbach, Futer, Kalfagianni, Lin and Stoltzfus \cite{Detal}. By considering the fact that the ribbon graphs of a link diagram extend Tait graphs, we are then led to ask in what ways the basic properties of Tait graphs extend to the more general ribbon graphs of a link diagram. 
\subsection{Tait graphs}\label{ss.tg}
Let $D\subset S^2$ be a link diagram. We  consider all link diagrams up to isotopy of $S^2$.
By a {\em checkerboard colouring} of $D$, we mean
the  assignment of  the colour black or white to each region of $D$ in such a way  that adjacent regions are assigned different colours. 
The {\em Tait sign}  of a crossing in a checkerboard coloured link diagram is an element of $\{+,-\}$ which is assigned to the crossing according to the following scheme:
\[ \begin{array}{ccc}
\raisebox{1mm}{\includegraphics[height=1cm]{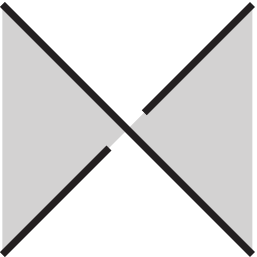}} & \quad& \raisebox{1mm}{\includegraphics[height=1cm]{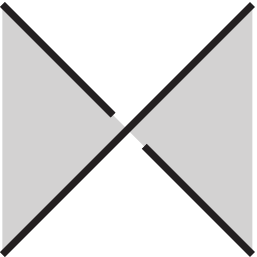}} \\
+ & & -
\end{array}.\]

 A {\em Tait  graph}, $\T(D)$, is a signed plane graph constructed from $D$ as follows: checkerboard colour the link diagram,  place a vertex in each black region and add an edge between two vertices whenever the corresponding regions of $D$ meet at a crossing; weight each edge of the graph with the Tait sign of the corresponding crossing. An example of a link diagram and its Tait graph is given in Figure~\ref{f.taitgraph}.

\begin{figure}
\[\includegraphics[width=3cm]{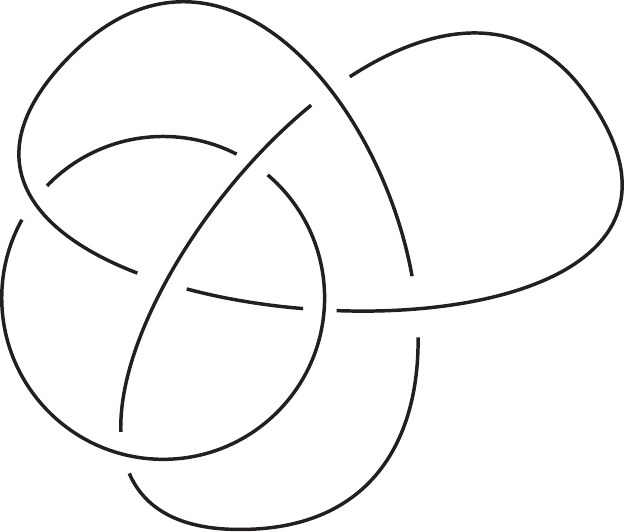} \raisebox{1cm}{\includegraphics[width=1cm]{arrow}} \quad
\includegraphics[width=3cm]{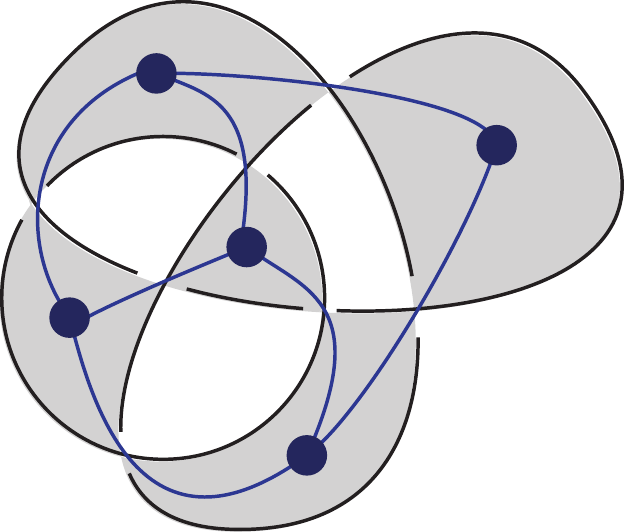}\quad
\raisebox{1cm}{\includegraphics[width=1cm]{arrow}}
\includegraphics[width=3cm]{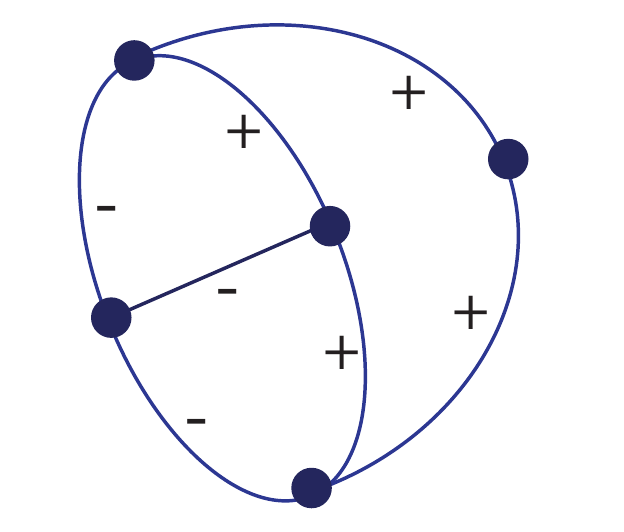}\]
\caption{Forming a Tait graph $\T(D)$ of a link diagram $D$. }
\label{f.taitgraph}
\end{figure}

Since there are two possible checkerboard colourings of  $D$, every diagram $D$ has exactly two  associated (signed) Tait graphs. 
The following facts about Tait graphs are well known and readily seen to be true.
\begin{enumerate}
\renewcommand{\theenumi}{T\arabic{enumi}}
\item  The two Tait graphs associated with a link diagram are geometric duals. 
\item  Every signed plane graph is the Tait graph of a link diagram.
\item  A Tait graph gives rise to a unique link diagram.
\end{enumerate}
For the first property, we recall that duality switches the sign of an edge.  For the third property, note that $D$ is recovered from $\T(D)$  by equipping the medial graph of $\T(D)$  with a crossing structure determined by  the Tait signs (this will be discussed in more detail in Subsection~\ref{ss.med}).

\subsection{The ribbon graphs of a link diagram}\label{ss.rgl}
In this subsection we describe  the extension  of  the concept of a Tait graph from \cite{Detal} (see also Turaev~\cite{Tu97}). In this construction a set of ribbon graphs, which includes the Tait graphs, is associated to a link diagram.

Let $D\subset S^2$ be a checkerboard coloured  link diagram.  Assign a unique label to each crossing of $D$. A {\em marked $A$-splicing} or a   {\em marked $B$-splicing} of a crossing $c$  is the replacement of the crossing with one of the schemes shown in Figure~\ref{f.spl}.

\begin{figure}
\begin{center}
\begin{tabular}{ccccc}
\includegraphics[height=1.2cm]{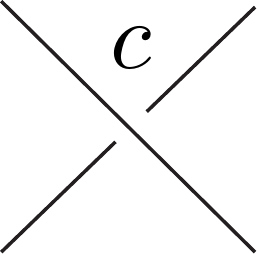}  & & \includegraphics[height=1.2cm]{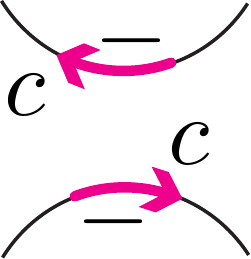} &&
\includegraphics[height=1.2cm]{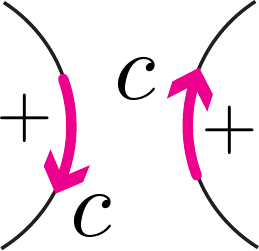} \\
 A crossing in $D$. &&  A marked $A$-splicing. && A marked $B$-splicing.
\end{tabular}
\end{center}
\caption{Splicing a crossing in a link diagram.}
\label{f.spl} 
\end{figure}


Notice that we decorate the two arcs in the splicing with signed 
labelled arrows that are chosen to be consistent with an arbitrary 
orientation of the sphere $S^2$. The labels of the arrows are 
determined by the label of the crossing, and the signs are 
determined by the choice of splicing.

A {\em state}, $\sigma$, of a link diagram is an assignment  of a marked $A$- or $B$-splicing to each crossing. Observe that a state is precisely an arrow presentation of a ribbon graph. We will  denote the ribbon graph corresponding to the state $\sigma$ of $D$ by $G(D,\sigma)$. These ribbon graphs are the ribbon graphs of a link diagram:
\begin{definition}[Dasbach et al. \cite{Detal}, Turaev \cite{Tu97}]
Let $D$ be a link diagram. Then the {\em set $ \G_s(D)$ of signed ribbon graphs associated with $D$} is defined by 
\[  \G_s(D) :=  \{ G(D,\sigma) \, | \,  \sigma \text{ is a marked state of } D   \}.  \]   
If $G\in \G_s(D)$, then we say that $G$ is a {\em signed ribbon graph of $D$}, and we  will also say that $G$ {\em presents} $D$.
\end{definition}
A example of a ribbon graph $G(D,\sigma)$ for a state $\sigma$ of a link diagram $D$ is given in Figure~\ref{f.rgoflink}.
 In this figure, for clarity, lines are used to emphasize the type of splicing, and the labelled arrows are omitted.

The (unsigned) ribbon graphs from \cite{Detal} are obtained by forgetting the signs in the constructions above. We  let $ \G(D)$ denote the set of ribbon graphs obtained from $ \G_s(D)$ by forgetting the signs, and call $G\in \G(D)$ a {\em ribbon graph of $D$}. For many of the questions we are interested in here, the signs of a ribbon graph of a link diagram are irrelevant, and so we work with the unsigned ribbon graphs.

\begin{figure}
\begin{center}
\begin{tabular}{ccc}
\raisebox{2mm}{\includegraphics[width=4cm]{l1}} &\raisebox{0cm}{\includegraphics[width=4cm]{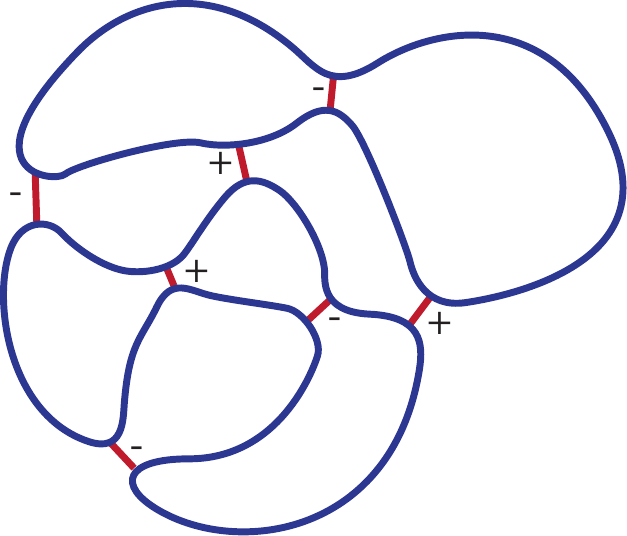}} &
\raisebox{-10mm}{\includegraphics[height=5cm]{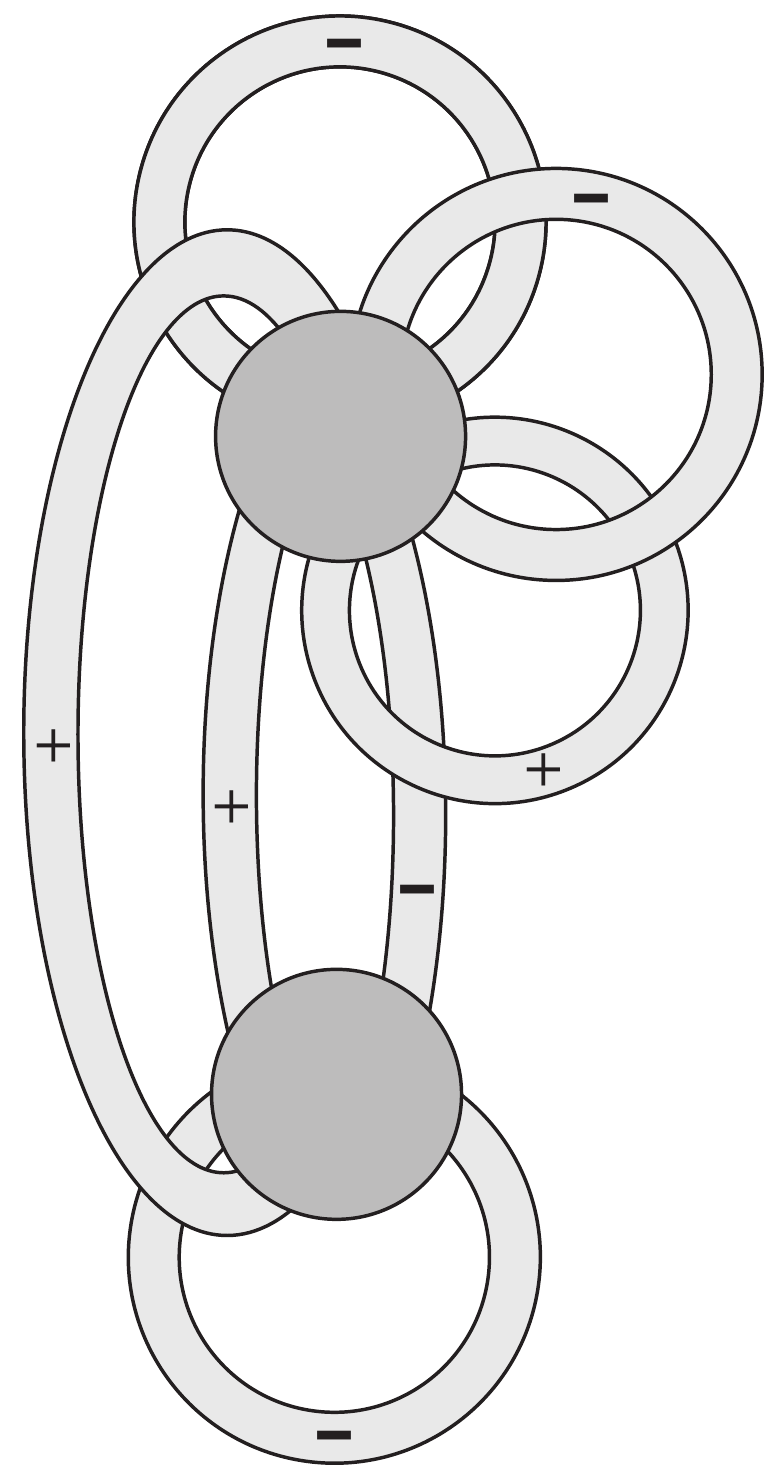}}
\\
A link diagram $D$. &A state $\sigma$ of $D$. & The ribbon graph $G(D,\sigma)$.
\end{tabular}
\end{center}
\caption{A link diagram $D$, a state $\sigma$ of  $D$ and the corresponding ribbon graph $G(D,\sigma)$. }
\label{f.rgoflink}
\end{figure}

The following proposition  states that the ribbon graphs of a link diagram do indeed generalize  Tait graphs.
\begin{proposition}[Dasbach et al. \cite{Detal}]\label{p.cb}
 Let $D$ be a link diagram, then both of the  Tait graphs of $D$  are  in $\G_s(D)$. 
\end{proposition}
\begin{proof}
Given  a checkerboard colouring of the diagram, we can recover one of the Tait graphs by, at each crossing, choosing the splicing in which the arcs follow the black faces of the diagram (i.e.  take a $B$-splicing at each positive crossing, and an $A$-splicing at each negative crossing). The other Tait graph is recovered by following the white faces (i.e. by taking  an $A$-splicing at each positive crossing, and a $B$-splicing at each negative crossing).
\end{proof}

\begin{remark}
For each diagram $D$, $\G_s(D)$ contains exactly two elements where all of the edges are of a single sign, either $+$ or $-$. These ribbon graphs are obtained by choosing the state that consists wholly of $B$-splicings, or wholly of $A$-splicings, respectively. If the sign on every edge of a ribbon graph is equal, then we may unsign it without losing any information. 
The all-$B$ and all-$A$ ribbon graphs from \cite{Detal} can be obtained from these two ribbon graphs by forgetting the signs.
These unsigned ribbon graphs  play a key role in the applications of the ribbon graphs of a link diagram to knot theory.
\end{remark}

\subsection{Three questions}
Our motivation comes from the  Properties~\ref{t1}-\ref{t3} of Tait graphs. We are interested in how these fundamental properties extend to the more general  ribbon graphs of a link diagram. 
 In particular we ask and answer the following three questions about the ribbon graphs of link diagrams. These questions are the natural extensions of Properties~\ref{t1}-\ref{t3} of Tait graphs.
\begin{enumerate}
\renewcommand{\theenumi}{Q\arabic{enumi}}
\item  How are the ribbon graphs of a link diagram related to each other?
\item   Which ribbon graphs arise as ribbon graphs of a link diagram?
\item  What is the relationship between link diagrams that are presented by  the same set signed ribbon graphs?  
\end{enumerate}
The answer to Question~\ref{q1} is known and the relation between the ribbon graphs of a link diagram is described in the following section.

Questions~\ref{q2} and \ref{q3} are more interesting than the analogous questions for Tait graphs, which are answered by \ref{t2} and \ref{t3}. For example, by considering all three crossing link diagrams, one can quickly verify that the ribbon graph shown on the right-hand side of Figure~\ref{f.e1sum} does not arise as the ribbon graph of a link diagram. Also, with a little more work, one can find examples of distinct link diagrams that are presented by the same ribbon graphs.

\section{A framework for answering Questions \ref{q1}-\ref{q3} }\label{s.pd}

Partial duality, introduced by Chmutov in \cite{Ch1}, is an extension of geometric   duality. 
Our interest in partial duality here lies in the fact that it provides a graph theoretical framework for addressing Questions~\ref{q1}-\ref{q3}. In this section we describe partial duality,  use it to answer Question \ref{q1}, and then to reformulate \ref{q2} and \ref{q3}.

\subsection{Partial duality}\label{ss.pd}
Loosely speaking, a partial dual is obtained by forming the geometric dual of an embedded graph only at a subset of its edges. Formally:

\begin{definition}[Chmutov \cite{Ch1}]\label{d.pd}
Let $G$ be a  ribbon graph and $A\subseteq E(G)$. Arbitrarily orient and label each of the edges of $G$. (The orientation need not extend to an orientation of the ribbon graph). The boundary components of the spanning ribbon subgraph $(V(G), A)$ of $G$ meet the edges of $G$ in disjoint arcs (where the spanning ribbon subgraph is naturally embedded in $G$). On each of these arcs, place an arrow which points in the direction of the orientation of the edge boundary  and is labelled by the edge it meets. 
The resulting marked boundary components of the spanning ribbon subgraph $(V(G), A)$ define an  arrow presentation. 
The  ribbon graph corresponding to this  arrow presentation is the {\em partial dual} $G^A$ of $G$.

If $G$ is a signed ribbon graph, then $G^A$ is also a signed ribbon graph whose signs obtained from $G$ by switching the sign of every edge in $A$. That is, if $e$ is an edge of $G$ with sign $\varepsilon$, then the corresponding edge in $G^A$ has sign $-\varepsilon$ if $e\in A$, and $\varepsilon$ if $e\notin A$.
\end{definition}

\begin{example}\label{ex.dex1} 
A  ribbon graph $G$ equipped with an arbitrary labelling and orientation of its edges is shown as Step 1 of Figure~\ref{f.e42}. For  this example we take $A=\{2,3\}$. The marked spanning ribbon subgraph $(V(G),A)$ is shown as Step 2 of Figure~\ref{f.e42}. The boundary components of this spanning ribbon subgraph define an arrow presentation, shown as Step 3. The corresponding  ribbon graph is shown as Step 4. This is the partial dual $G^{\{2,3\}}$ of $G$.


\begin{figure}
 \begin{center}
\begin{tabular}{ccc}
 \includegraphics[width=5cm]{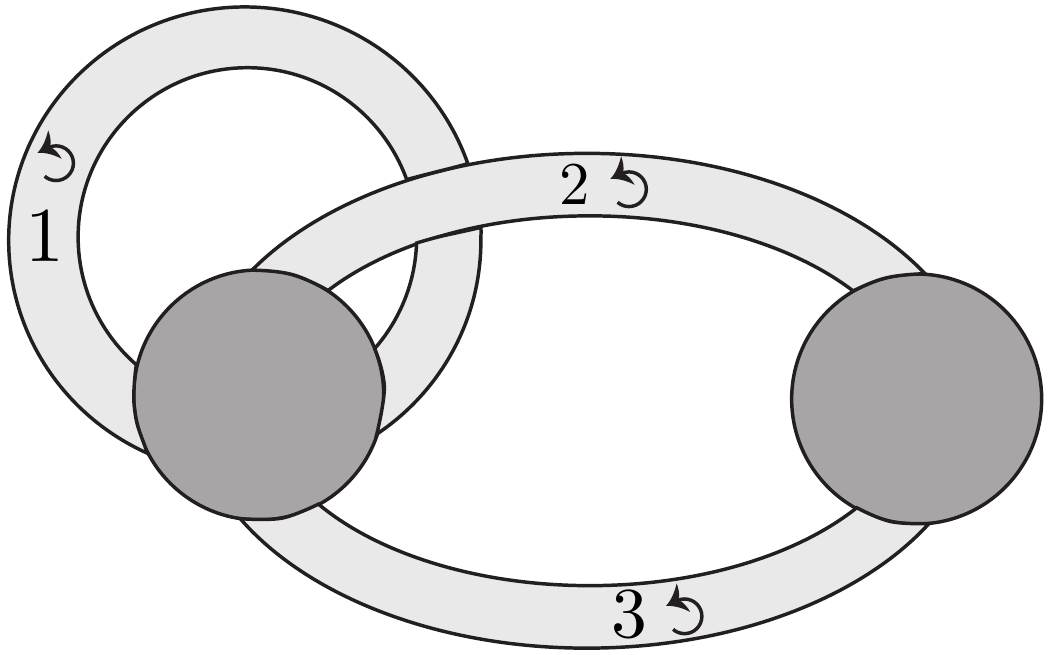}  & \hspace{1cm} & \includegraphics[width=5cm]{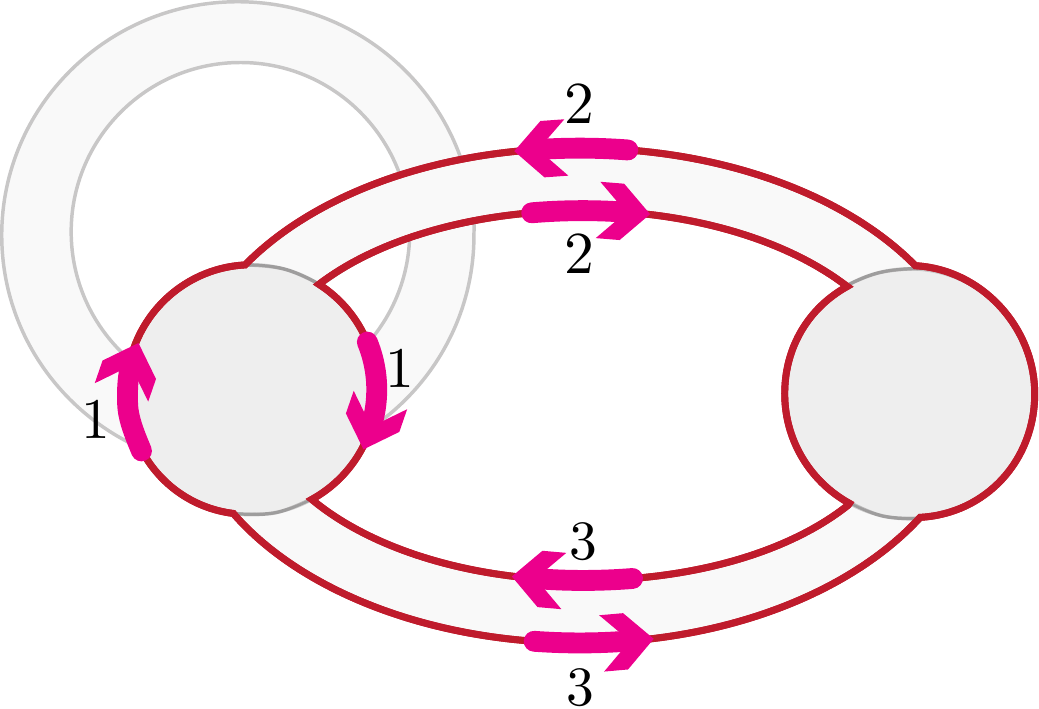}   \\
Step 1.  & &Step 2. \ 
 \end{tabular}
 \end{center}

 \begin{center}
\begin{tabular}{ccc}
\includegraphics[width=5cm]{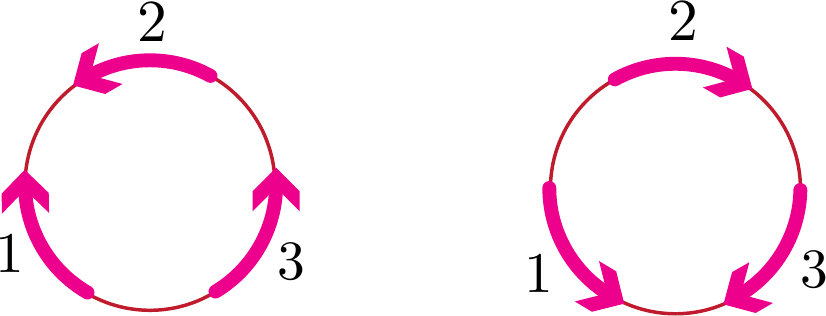} & \hspace{1cm} & \includegraphics[width=5cm]{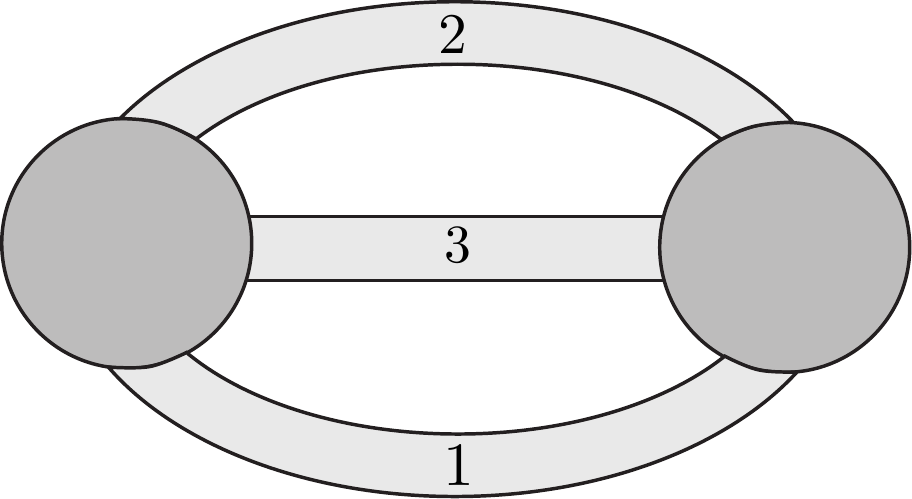}      \\
 Step 3.& & Step  4.
\end{tabular}
\end{center}
\caption{Forming the partial dual of a ribbon graph as in Example~\ref{ex.dex1}.}
\label{f.e42}
\end{figure}

\end{example}

Additional examples of partially dual graphs can be found, for example,  in \cite{Ch1, Mo3, Mo4, VT}.

We will use the following basic properties of partial duals.
\begin{proposition}[Chmutov~\cite{Ch1}]\label{p.pd2}
Let $G$ be a ribbon graph and $A, B\subseteq E(G)$.  Then 
\begin{enumerate}
\item \label{p.pd2.1} $G^{\emptyset}=G$;
\item \label{p.pd2.2}  $G^{E(G)}=G^*$, where $G^*$ is the geometric dual of $G$;
\item \label{p.pd2.3} $(G^A)^B=G^{A\Delta B}$, where $A\Delta B := (A\cup B)\backslash  (A\cap B)$ is the symmetric difference of $A$ and $B$;
\item partial duality acts disjointly on components, i.e. $ (P\sqcup Q)^A = (P^{A \cap E(P)} ) \sqcup (Q^{A \cap E(Q)})$.
\end{enumerate}
\end{proposition}

It is worthwhile noting that Item~\ref{p.pd2.3} allows for partial duals to be formed one edge at a time.

For a set $X$ and an element $x$, we write $X\backslash x$  for $X\backslash \{x\}$.

\subsection{Partial duals and the graphs of links:  solution to Question~\ref{q1}}

Tait graphs are geometric duals of one another (Property~\ref{t1}). The following proposition, which appeared  in  \cite{Ch1} and implicitly in \cite{Mo2}, states that all of the ribbon graphs in $\G_s(D)$ are partial duals of  one another. In particular, since Tait graphs are in $ \G_s(D)$, the ribbon graphs of a link diagram are all partial duals of  its Tait graphs. 

\begin{proposition}[Chmutov \cite{Ch1}]\label{p.pd3} Let $D$ be a link diagram. Then a signed ribbon graph 
$G$ presents $D$ (that is, $G\in \G_s(D)$) if and only if $G$ is a partial dual of a Tait graph of $D$.
\end{proposition}
\begin{proof}
Taking the partial dual with respect to an edge $e$ 
changes an arrow presentation as follows:
\[  \raisebox{-5mm}{   \includegraphics[height=1cm]{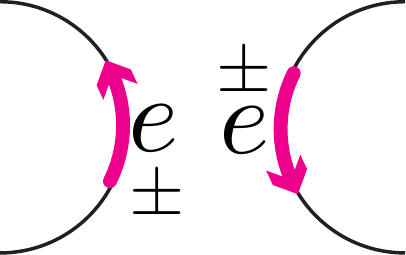}}  \raisebox{-1mm}{\includegraphics[width=1cm]{arrow}}    \raisebox{-5mm}{\includegraphics[height=1cm]{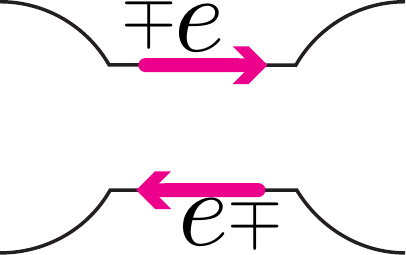}} \;\raisebox{-5mm}{.}\]
This is readily seen to correspond to the change in the state of a link diagram caused by switching between $A$- and $B$-splicings at the crossing corresponding to $e$ as in Figure~\ref{f.spl}.
\end{proof}

The following corollary provides a reformulation our motivating questions \ref{q1}-\ref{q3} in terms of the graph theoretical language of partial duals of plane graphs. It is in this language that we work to answer the outstanding questions, \ref{q2} and \ref{q3}.
\begin{corollary}\label{c.pd6}
$G$ is the (signed or unsigned) ribbon graph of a link diagram if and only if it is a partial dual of a (signed or unsigned) plane graph.
\end{corollary}
\begin{proof}
The corollary follows immediately from Proposition~\ref{p.pd3} and Property~\ref{t2}. 
\end{proof}

\subsection{Partial duals of plane graphs: a graph theoretical framework for~\ref{q2} and~\ref{q3}}\label{ss.med}

Having described a solution to \ref{q1}, we turn out attention to  \ref{q2} and~\ref{q3}. 

Corollary~\ref{c.pd6} provides the graph theoretical formulation of question~\ref{q2} in which we will answer it.
\begin{enumerate}
\setcounter{enumi}{1}
\renewcommand{\theenumi}{Q\arabic{enumi}${}^{\prime}$}
\item  \label{q2'} Which ribbon graphs are the partial duals of plane graphs?  
\end{enumerate}

Having found a workable reformulation of \ref{q2}, we turn our attention to \ref{q3}. In order to tackle this question we need to understand how, if we are given a signed ribbon graph $G$,  we can recover the link diagrams that have $G$ as their signed ribbon graph.   We  begin by describing how this is done for plane ribbon graphs. 

Let $G$ be a signed plane ribbon graph. We can construct a link diagram  on $G$ by drawing the following configuration on each of its edges, and connecting the configurations by following the boundaries of the vertices of $G$:
\[    \includegraphics[height=1cm]{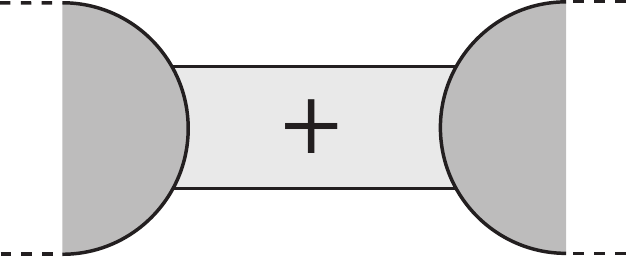}  \raisebox{4mm}{\includegraphics[width=1cm]{arrow}}  \includegraphics[height=1cm]{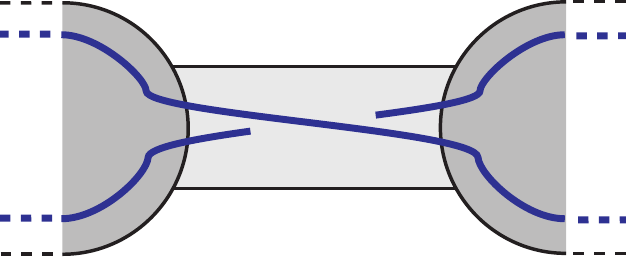}  \hspace{3mm}\raisebox{4mm}{and}  \hspace{3mm}
 \includegraphics[height=1cm]{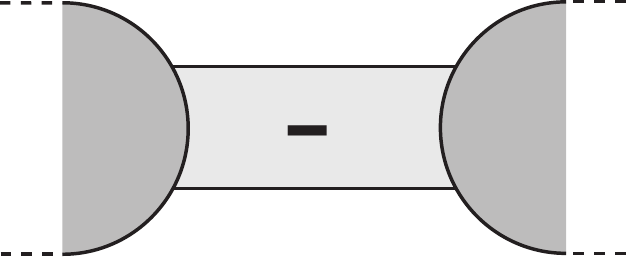}   \raisebox{4mm}{\includegraphics[width=1cm]{arrow}} \includegraphics[height=1cm]{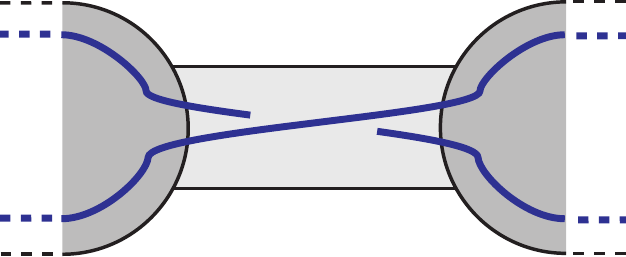}  \;   .\]
Since $G$ is topologically a punctured sphere,  capping off the punctures of $G$ then gives a link diagram $D(G)$ on the $2$-sphere. We will call $D(G)$ the {\em link diagram associated with $G$}.
Clearly, the two Tait graphs of $D(G)$ are $G$ and $G^*$, and $D(\T(D))=D$.

Now suppose that $G$ is a (not necessarily plane) signed ribbon graph and that it is the ribbon graph of some link diagram. We want to find the set $\calD(G)$ of link diagrams that have $G$ as a signed ribbon graph.  By Proposition~\ref{p.pd3}, $G$ presents a link diagram $D$ if and only if it is the partial dual of a Tait graph of $D$. Then, since every signed plane graph is a Tait graph of a link diagram, it follows that 
\begin{equation}\label{e.cald}
  \calD(G)=\{ D(G^A) \; | \;  A\subseteq E(G) \text{ and } G^A \text{ is plane}   \}. 
  \end{equation}
Although there is a unique link diagram associated with a pair of Tait graphs, in general, a ribbon graph can have many plane partial duals, and  $ \calD(G)$ can contain more than one distinct link diagram. 

Question~\ref{q3} therefore asks for the relation between $D(G)$ and $D(H)$, where $G$ and $H$ are both partially dual plane graphs. This leads to the graph theoretical question:
\begin{enumerate}
\setcounter{enumi}{2}
\renewcommand{\theenumi}{Q\arabic{enumi}${}^{\prime}$}
\item  \label{q3'} If $G$ and $G^A$ are both plane graphs, how are they related to each other?  
\end{enumerate}

We now turn our attention to the solutions of \ref{q2'} and \ref{q3'}. In Section~\ref{s.1-sums} we discuss $1$-sums of ribbon graphs and introduce \pbs - a simple variation on the notion of  separability of a graph. These give a natural way to decompose a ribbon graph into plane graphs. Then, in Section~\ref{s.char}, we answer Question~\ref{q2'}, and hence \ref{q2}, by showing that partial duals of plane graphs are completely characterized by the existence of \pbst. 
We then apply this result in Section~\ref{s.join}  to answer \ref{q3'}, which, in turn, is used to answer \ref{q3} in Section~\ref{s.knots}.

\begin{remark}
The relation between the ribbon graphs of a link diagram was a primary motivation for the introduction of partial duality by Chmutov in \cite{Ch1}. 
Similarly, Tait graphs and their relation to medial graphs also provided the primary motivation of twisted duality, a generalization of duality introduced  in \cite{EMM}. 
If $G$ and $H$ are embedded graphs and $G_m$ and $H_m$ are their embedded medial graphs, then $G$ and $H$ must be Tait graphs of  $G_m$ and $H_m$. It then follows that $G_m$ and $H_m$ are equal as embedded graphs if and only if $G$ and $H$ are geometric duals. Thus we can take the point of view that  the relation of equality as embedded graphs generates the relation of geometric duality. This point of view suggests that other concepts of equality of embedded graphs will generate other concepts of duality.  In   \cite{EMM} it was show that equality as abstract graphs  generates twisted duality, and that equality as combinatorial maps generates partial duality.  In addition, \cite{EMM} contains a way to construct the Tait graphs of non-checkerboard colourable $4$-regular embedded  graph.
\end{remark}

\section{$1$-sums and decompositions of ribbon graphs}\label{s.1-sums}

In this section we discuss separability and $1$-sums of ribbon graphs. These  are   natural extensions of the corresponding operations for graphs.  We will go on to introduce the concept of a \pb and of a \jpbt.  We will see  that these types of separations of a ribbon graph give  characterizations of ribbon graphs that have plane partial duals,  providing a connection between the genus of a partial dual and  separability.

\subsection{$1$-sums of ribbon graphs}\label{ss.sum}

Let $G$ be a ribbon graph, $v\in V(G)$, and $P$ and $Q$ be non-trivial ribbon subgraphs of $G$. Then $G$ is said to be the {\em $1$-sum} of $P$ and $Q$, written $P\oplus Q$, if $G=P\cup Q$ and $P\cap Q=\{v\}$. The $1$-sum is said to {\em occur} at the vertex $v$, and $P$ and $Q$ are the {\em $1$-summands}. See Figure~\ref{f.join}.
Note that we do not require the ribbon graphs $G$, $P$ or $Q$ to be connected.

\begin{figure}
\begin{tabular}{ccccc}
\includegraphics[height=2cm]{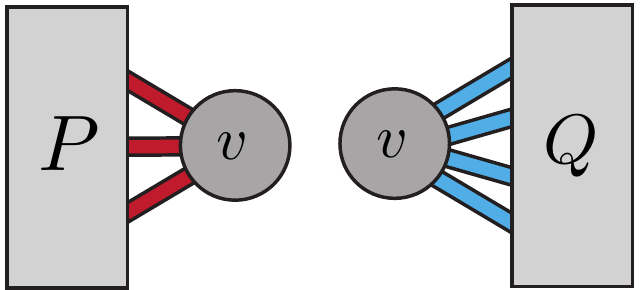} & \hspace{0cm} & \includegraphics[height=2cm]{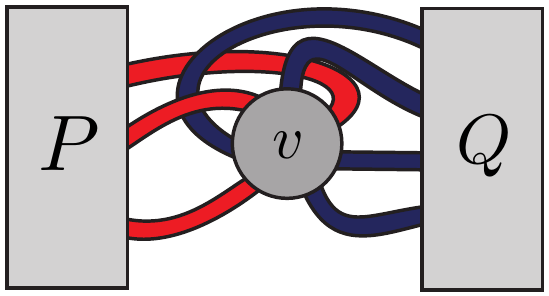} 
& \hspace{0cm} & \includegraphics[height=2cm]{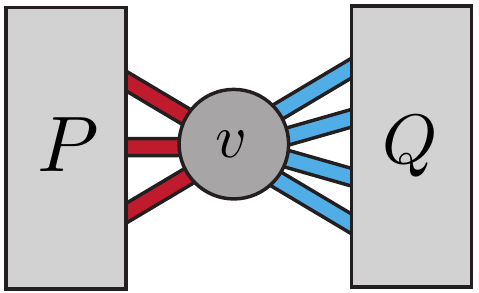}  \\
Ribbon graphs $P$ and $Q$. && A $1$-sum  $P\oplus Q$ && A join $P\vee Q$.
\end{tabular}
\caption{A $1$-sum and a join of two ribbon graphs.}
\label{f.join}
\end{figure}

\begin{example}\label{e.1sum}
Two examples of a $1$-sum $G=P\oplus Q$, and their $1$-summands are shown in Figure~\ref{f.e1sum}. 
Observe  that while example on the left can be written as a $1$-sum of two plane ribbon graphs, the example on the right cannot (although it is the $1$-sum of three plane ribbon graphs). 

\begin{figure}
\begin{center}
\begin{tabular}{c|c}
\includegraphics[height=3cm]{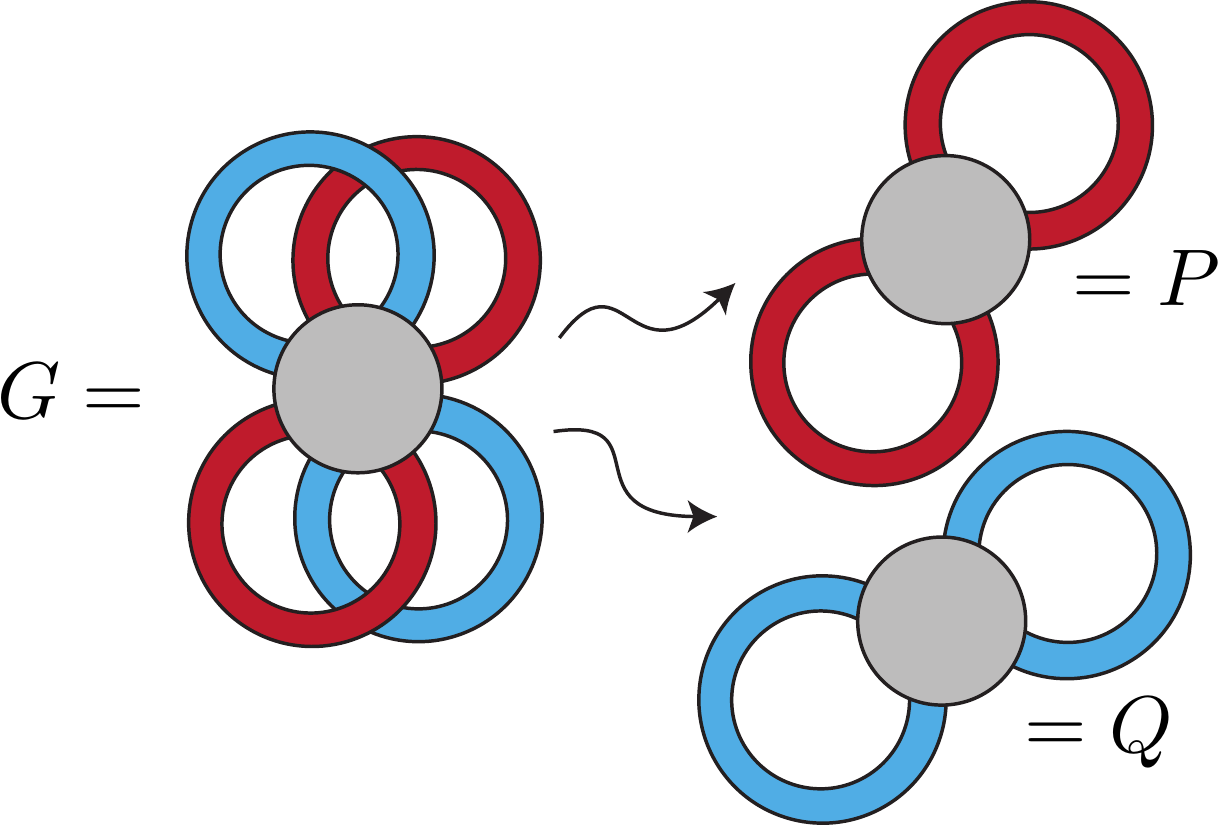} \hspace{6mm} & \hspace{6mm}
\includegraphics[height=3cm]{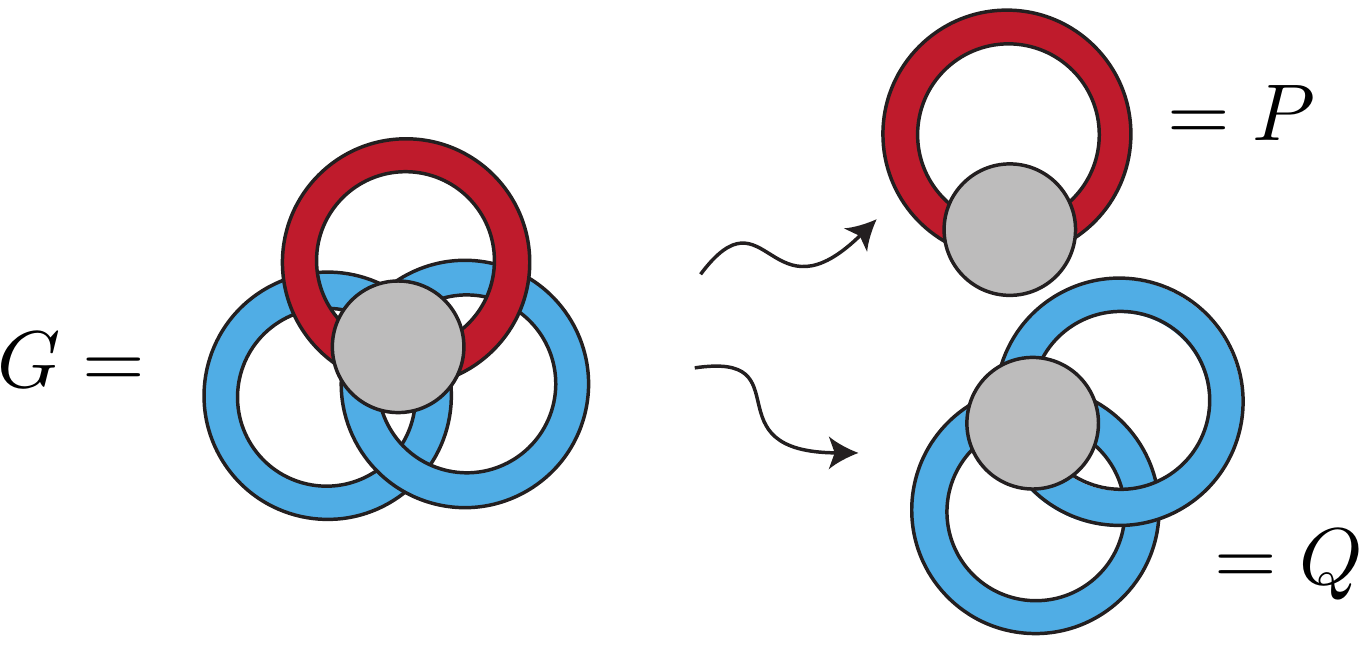} 
\end{tabular}
\end{center}
\caption{Two examples of  $1$-sums. The example on the left is a \pbt, while that on the right is not.}
\label{f.e1sum}
\end{figure}

\end{example}

Although  $G  =P\oplus Q$ if and only if this identity holds for their underlying abstract graphs, it is important to keep in mind that the topologies of $G$, $P$ and $Q$ can be quite different (for example, if $P$ consists of one vertex and one edge and is orientable, then $P\oplus P$ could be of genus $0$ or $1$ depending on the particular $1$-sum). Also, if  $G=P_1\oplus Q_1=P_2\oplus Q_2$  with the $1$-sum occurring  at the same vertex then the topological properties of $P_1$ and $Q_1$, and $P_2$ and $Q_2$ can be quite different (for example, the   genus $2$, orientable ribbon graph with one vertex and four edges with cyclic order $(a\,b\,a\,b\,c\,d\,c\,d)$ can be written as a $1$-sum of  two subgraph of genus $1$, or of two subgraphs of genus $0$, or of one subgraph of genus $1$ and one of genus $0$). 
   
 Here we are interested in expressions of $G$ as $1$-sums of ribbon graphs.  
We  say that $G$ can be {\em written as a sequence of $1$-sums} if $G$ contains subgraphs $H_1, \ldots ,H_l$ such that
\begin{equation}\label{e.seq}
G=   H_1 \oplus H_2\oplus H_3   \oplus \cdots\oplus H_l := (\cdots ((H_1 \oplus H_2)\oplus H_3)   \oplus \cdots\oplus H_l). 
\end{equation}
Observe that, in the expression \eqref{e.seq}, the $H_i$'s are  non-trivial ribbon subgraphs that cover $G$; for each $i\neq j$, $H_i$ and $H_j$ have  at most one vertex in common; and that if a $1$-sum occurs at a vertex $v$ in the sequence, then $v$ is a separating vertex (but not necessarily a cut vertex) of the underlying abstract graph of $G$.
We consider sequences of $1$-sums for $G$  to be {\em equivalent} if they differ only in the order of $1$-summation. (Note that only some reorderings of the $1$-summands in a sequence of $1$-sums are possible.) We consider all sequences of $1$-sums up to this equivalence.

We now introduce our main structural decomposition of a ribbon graph: a \pbt. The intuitive idea behind Definition~\ref{d.d} is that $G$ has a \pb giving rise to $P_A$ and $Q_A$ if both $P_A$ and $Q_A$ are plane and $G$ can be obtained by $1$-summing the components of $P_A$ and $Q_A$ to each other in such a way that every $1$-sum occurs at a distinct vertex and involves a component of $P_A$ and a component of $Q_A$.

\begin{definition}\label{d.d} 
Let $G=(V,E)$ be a connected ribbon graph and $A\subseteq E(G)$.  In addition, let  $P_A=G|_A$ be the ribbon subgraph induced by $A$, and $Q_A=G|_{A^c}$ be the ribbon  subgraph induced by $A^c=E(G)\bs A$.   We say that $A$ defines a {\em \pb} if either
\begin{enumerate}
\item  $A=E$ or $A=\emptyset$, and $P_A$ and $Q_A$ are both plane (in which case the \pb is said to be {\em trivial}); or,
\item $P_A$ and $Q_A$ are  plane and $G$ can be written as a sequence of $1$-sums of the components of $P_A$ and $Q_A$ such that every $1$-sum occurs at a different vertex of $G$ and involves a component of  $P_A$  and a component of $Q_A$.
\end{enumerate}

If $G$ is not connected, we say that $A$ defines a \pb if the restriction of $A$ to any component of $G$ defines a \pb of that component.
\end{definition}

The {\em length} of a non-trivial \pb  is the length of its sequence of $1$-sums, and the length of a trivial \pb is defined to be $1$. We will  say that $G$ admits a \pb if $A$ defines a \pb for some $A\subseteq E(G)$. 

Note the roles of $P_A$ and $Q_A$ in a \pb are interchangeable, and  that $A$ defines a \pb of $G$ if and only if $E(G) \backslash  A$ also defines one.  

As the set $A$ in a \pb is completely determined by $P_A$, we can, and will, specify \pbs by giving $P_A$ or $Q_A$, referring to the  triple $(G,P_A,Q_A)$ as a \pb of $G$.

\begin{example}\label{e.deco1}
A \pb of a ribbon graph  is illustrated below.
\[\includegraphics[height=3.4cm]{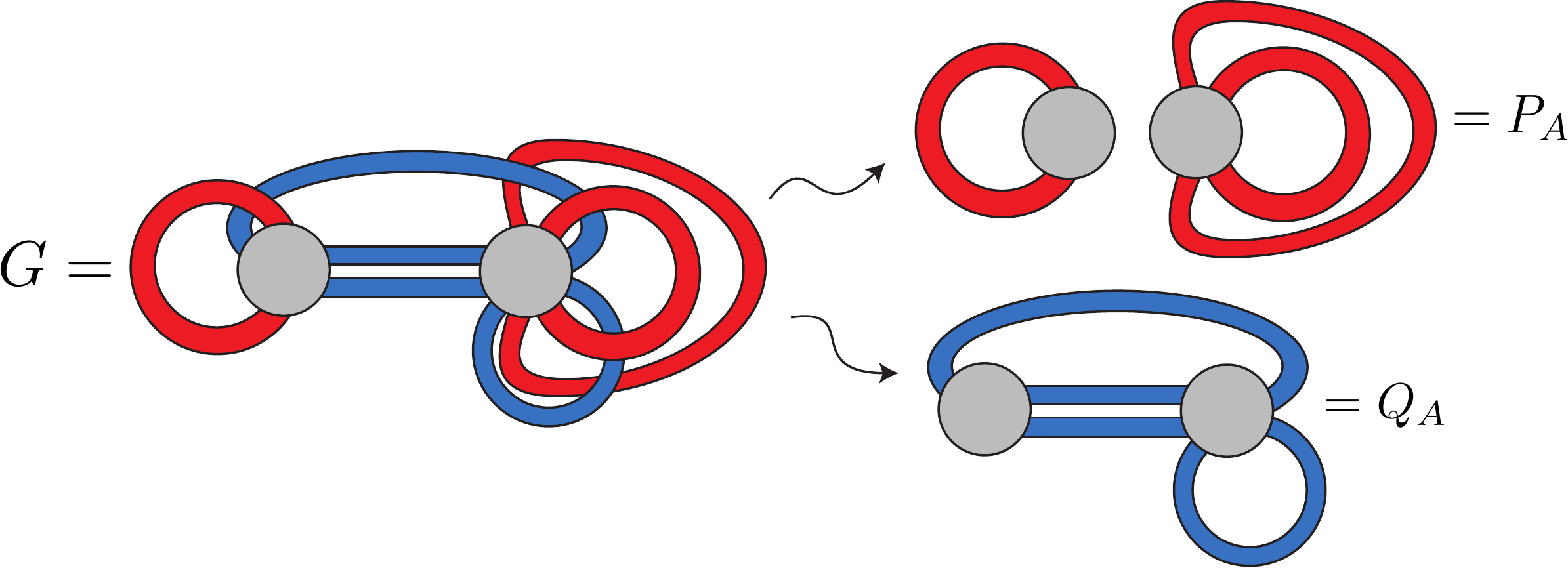}\]
Note that $G$ only admits one other \pbt: that defined by $E(G)\backslash A$.
\end{example}

\begin{example}\label{e.deco4}
A \pb  is illustrated below. 
\[\includegraphics[height=3.8cm]{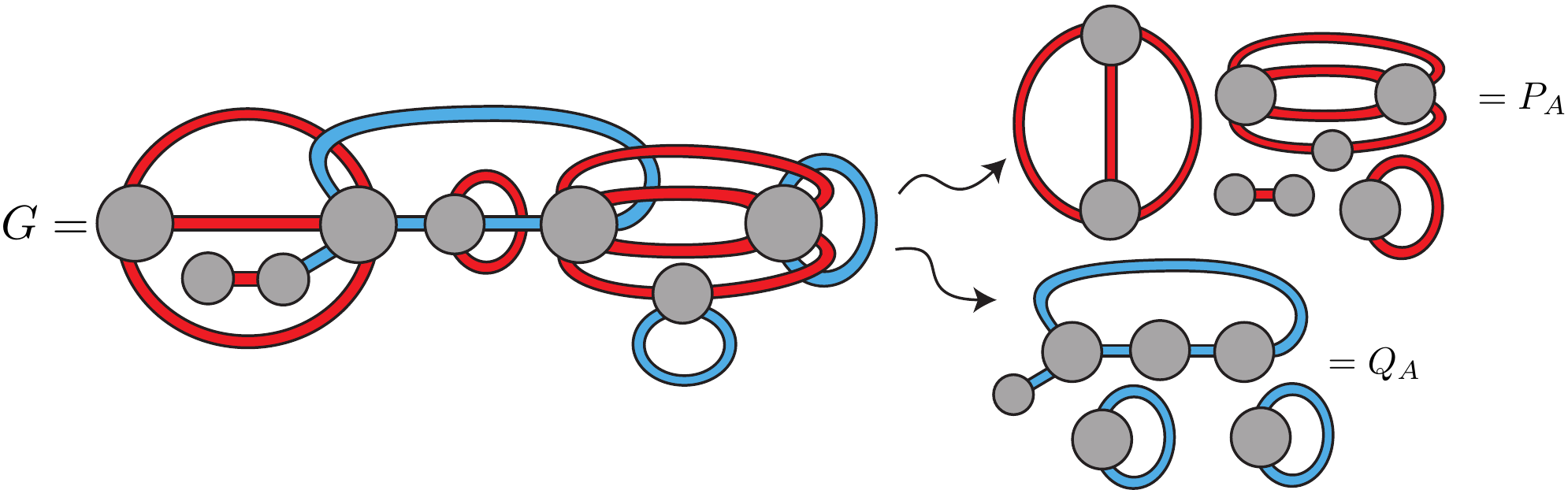}\]
Note that $G$ admits several other  \pbst, two of which are shown in Figure~\ref{f.e7}.
\end{example}
The $1$-sum on the left in Figure~\ref{f.e1sum} can also be regarded as a \pbt. 
The ribbon graph on the right in this example does not admit a \pbt.

\subsection{Joins}

 The join operation is a simple, special case of the $1$-sum operation.

\begin{definition}
Suppose $G=P\oplus Q$ with the $1$-sum occurring at $v$. If  
there is an arc on the  boundary of $v$ with the property that all edges of $P$ incident to $v$ intersect it on this arc, and that no edges of $Q$ intersect this arc, then  $G$ is said to be the {\em join} of $P$ and $Q$, written $G=P\vee Q$. (See Figure~\ref{f.join}.)
\end{definition}
The join  is also known as the ``one-point join'', a `` map amalgamation'' and the ``connected sum'' in the literature.


We will see in Section~\ref{s.join} that \jpbst, defined below, provide a characterization of plane  partial duals of plane graphs.
\begin{definition}\label{d.js}
Let $G=(V,E)$ be a ribbon graph and $A\subseteq E(G)$.    We say that $A$ defines a {\em \jpb}
if $G$ can be written as  $G=H_1\vee \cdots \vee H_l$, where $l\geq 1$, such that 
\begin{enumerate}
\item each $H_i$ is plane,
\item $A= \bigcup_{i\in I} E(H_i)$, for some $I\subseteq \{1, \ldots , l\}$.
\end{enumerate}
\end{definition}

It is worth emphasizing that the joins in Definition~\ref{d.js} need not occur at distinct vertices. Also the $H_i$ may themselves be joins of ribbon graphs.

Just as with \pbst, we can specify \jpbs by giving $P_A=G|_A$ and $Q_A=G|_{A^c}$.

\begin{example}\label{e.deco3}
A \jpb  is illustrated below. 
\[\includegraphics[height=3cm]{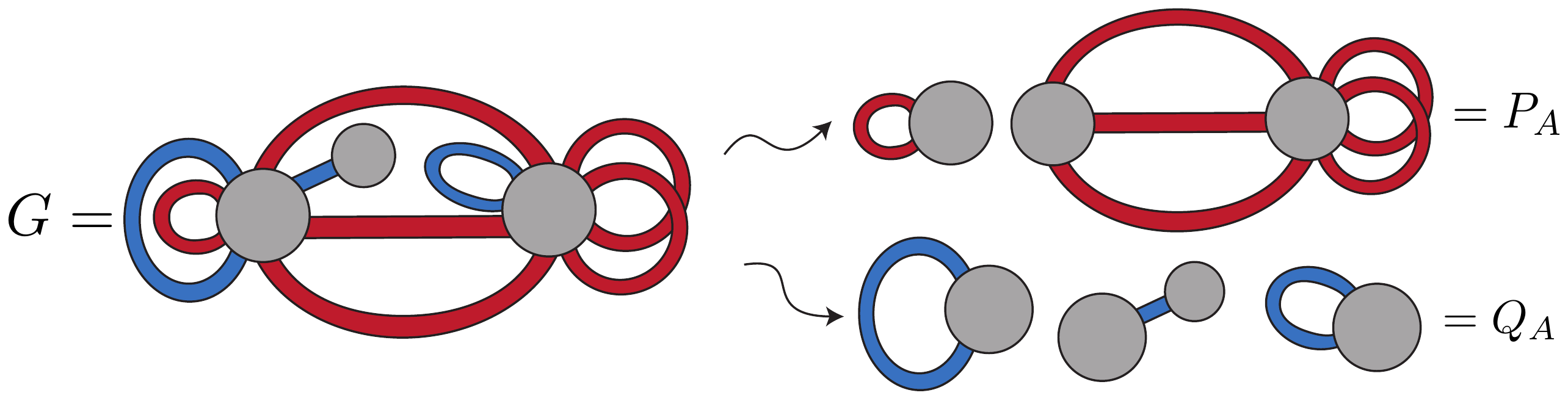}\]
\end{example}

Note that, since genus is additive under the join operation, a ribbon graph $G$ admits a \jpb if and only if it is plane. 

The following lemma relates \pbs and \jpbs for plane graphs. This result will be useful later.
\begin{proposition}\label{p.v1r}
Let $G$ be a plane ribbon graph. A subset $A\subseteq E(G)$ defines a \pb if and only if it  defines a \jpbt.
\end{proposition}
\begin{proof} It is enough to prove the proposition for  connected ribbon graphs, so assume that $G$ is connected.

We prove necessity by induction on the number of edges. If $|E(G)|=0$ or $1$ the result is trivial. Suppose that the assertion holds for all ribbon graphs with fewer than $k$ edges. Now suppose that $|E(G)|=k$ and $A\subseteq E(G)$ defines a \pbt. If $A$ defines a trivial \pb the result is obvious. If  $A$ defines a non-trivial \pbt, 
suppose that a $1$-sum occurs at a vertex $v$ in the \pbt. Since $G$ is  plane and a $1$-sum  occurs at $v$ it follows that the boundary of $v$ can be partitioned into two arcs $p$ and $q$ which have the properties that an edge of $G$ meets $p$ and an edge meets $q$; and that any path between a half-edge of $G$ that is incident to $p$ and a half-edge that is incident to $q$ must pass though $v$. It follows from this that we can write $G=K\vee J$ with the join occurring at $v$. Furthermore, $A\cap E(K)$ defines a \pb of $K$, and  $A\cap E(J)$ defines a \pb of $J$. By the inductive hypothesis,  $A\cap E(K)$ defines a \jpb of $K$, and  $A\cap E(J)$ defines a \jpb of $J$, and it follows that $A$ defines a \jpb of $G$.

Conversely, let  $G=H_1\vee \cdots \vee H_l$ and $A= \bigcup_{i\in I} E(H_i)$. Then, since each  $H_i$ is plane and the join of plane ribbon graphs is plane, setting 
$P_A$ to be the spanning ribbon subgraph of $G$ induced by $A$, and $Q_A$  to be the ribbon subgraph induced by  $E(G)\backslash A$ determines the required  \pbt. 
\end{proof}

\subsection{Partial duals of $1$-sums}
We are especially interested in  connections between $1$-sums and partial duals. In general, partial duality does not preserve $1$-sums. However, in this paper we will see that  $1$-sums and partial duals are closely related. This subsection contains two key  lemmas on partial duals of $1$-sums.

\begin{lemma}\label{l.pchar3}
Let $G$ be a ribbon graph such that $G=P\oplus Q$, where $Q$ is a plane ribbon graph. Then the genus of $ G^{E(Q)}$ is equal to  the genus of $P$.
Moreover, if the $1$-sum occurs at $v$, then  every vertex in $V(P) \backslash  v$ is also a vertex of $G^{E(Q)}$.
\end{lemma}

\begin{proof} Without loss of generality, assume that $G$ is connected.
Suppose that the $1$-sum occurs at the vertex $v$. Let $v_P$ and $v_Q$ denote the copies of $v$ in $P$ and $Q$, respectively.  $G  =P\oplus Q$ is obtained from $P$ and $Q$ by identifying  $v_P$ and $v_Q$. Let $\phi:\thinspace v_Q\rightarrow v_P$ denote the identifying map, and let $\phi|_{\partial}$, denote the restriction of $\phi$ to the boundary. 

We begin by  cellularly embedding $P$ into a surface $\Sigma$, and  $Q$ into the $2$-sphere, $S^2$, in such a way that $v_Q$ consists of the southern hemisphere, and  the equator of $S^2$ is the boundary of $v_Q$.  (For an example of these embeddings see  Figure~\ref{f.suff}.)

The partial dual $ G^{E(Q)}=(P\oplus Q)^{E(Q)}$ can be formed from the embeddings of $P$ and $Q$ in the following way. 
\begin{enumerate}
\item Form the dual $Q^*\subset S^2$ of $Q\subset S^2$. (Note that $Q^*$ lies entirely in the northern hemisphere of $S^2$.)
\item Delete the vertex $v_P\subseteq P\subset \Sigma$. Also delete the southern hemisphere of $S^2$ to obtain an embedding of $Q^*$ in a disc $\mathfrak{D}$. \item  Identify the boundaries of $\mathfrak{D}$ and $\Sigma \bs v_P$ using   $\phi|_{\partial}$.  (Note the boundaries of  $\Sigma \bs v_P$ and $v_P$ are equal, and that the boundaries of  $\mathfrak{D}$ and $v_Q$ are equal.)
\end{enumerate} 
 (For an example of the formation of the partial dual in this way see Figure~\ref{f.suff}.)
 This results in a cellular embedding of $ G^{E(Q)}=(P\oplus Q)^{E(Q)}$ in $\Sigma$.
Finally, since $P$ and $G^{E(Q)}$ can both be cellularly embedded in  $\Sigma$, they must be of the same genus. It is obvious from this construction that   every vertex in $V(P) \backslash  v$ is also in $G^{E(Q)}$.
\end{proof}

\begin{figure}
\begin{center}
\begin{tabular}{ccc}
 \includegraphics[height=3cm]{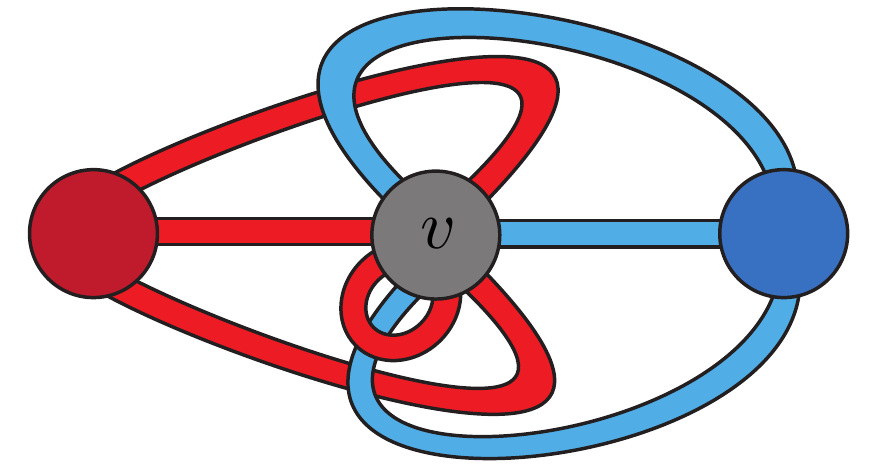}
&& \includegraphics[height=3cm]{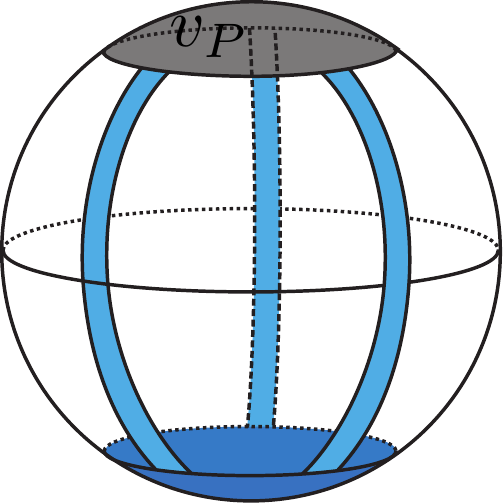}
\\
 $G=P\oplus Q$  && $P$
 \\
 &&\\
\includegraphics[height=3cm]{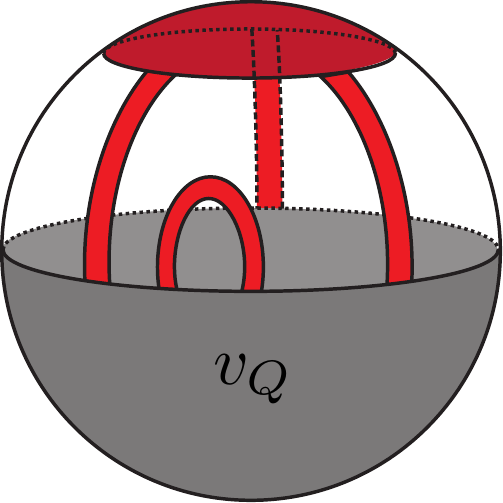} &&  \includegraphics[height=3cm]{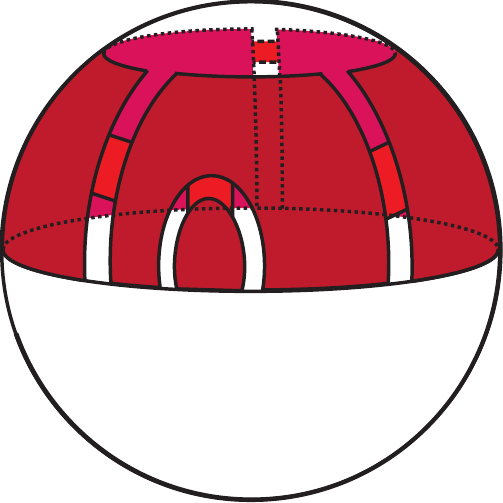}
\\
 $Q$  && $Q^*$
 \\
 &&
 \\
 \includegraphics[height=3cm]{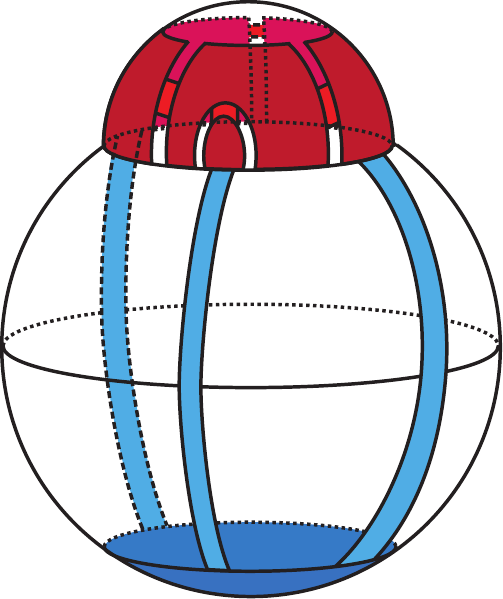}
&&\includegraphics[height=3cm]{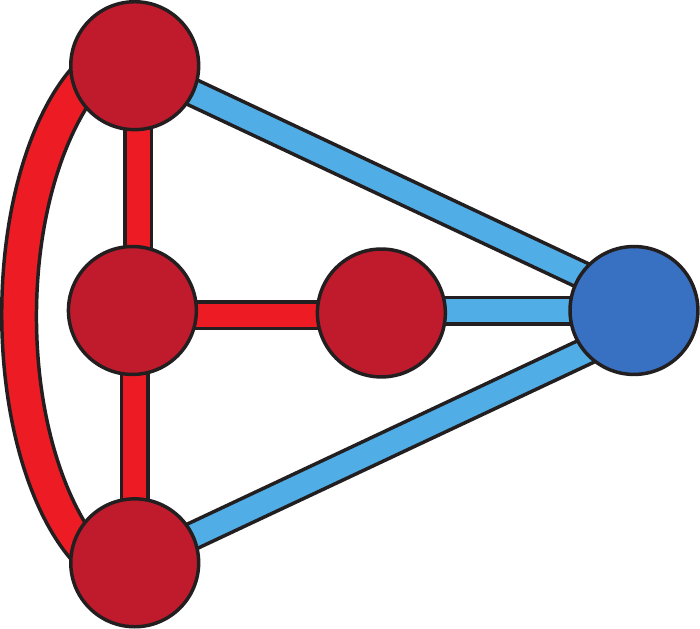} 
\\
 $G^{E(Q)}$  && Redrawing $G^{E(Q)}$
\end{tabular}
\end{center}
\caption{An example of the construction used in the proof of Lemma~\ref{l.pchar3}.}
\label{f.suff}
\end{figure}

Although, in general, partial duality does not preserve $1$-sums, it does  preserve joins.
To see this, suppose  $G=P\vee Q$ with the join occurring at the vertex $v$.   
The boundary of $v$ can be partitioned into two arcs $p$ and $q$ such that the edges of $P$ intersect $v$ on $p$, and the edges of $Q$ intersect $v$ on $q$. Let $a$ and $b$ be the two points on  the boundary of $v$ where $p$ and $q$ meet.

In the formation of the partial dual $G^A$, following Definition~\ref{d.pd}, there is a boundary cycle $\alpha$ of $(V(G),A)$ that contains both $a$ and $b$ and this is the only boundary cycle of $(V(G),A)$ that can contain marking arrows that have labels belonging to both $E(P)$ and $E(Q)$. Moreover, the points $a$ and $b$ partition $\alpha$ into two arcs $p'$ and $q'$, with $p'$ containing no  $E(Q)$-labelled arrows, and $q'$ containing no  $E(P)$-labelled arrows. 
The boundary cycle $\alpha$ corresponds to a vertex $v'$ of $G^A$ and $p'$ corresponds to an arc on a vertex of $P$. It is then clear  that $v'$ and $p'$ determine an expression of  $G^A$ as the join $P^{E(P)\cap A}\vee Q^{E(Q)\cap A}$. 
This discussion  is summarized by the  following lemma. 
\begin{lemma}\label{sum5}
Suppose $G$ is a ribbon graph that can be expressed as the join $G=P \vee Q$, and that $A\subseteq E(G)$. Then 
\[  G^A=  ( P \vee Q)^A = P^{A'}\vee Q^{A''},   \]
where $A':=A\cap E(P)$, $A'':=A\cap E(Q)$, and the joins act naturally with respect to partial duality.
\end{lemma}

\section{Partial duals of plane graphs and \pbst}\label{s.char}

In this section we characterize partial duals of plane graphs in terms of \pbst, which is the  first of our main results. We then go on to determine how all of the \pbs of a ribbon graph into two plane graphs are related to one another.

\subsection{A characterization of the partial duals of a plane graph}

\begin{theorem}\label{t.pchar1}
Let $G$ be a ribbon graph and $A\subseteq E(G)$. Then $G^A$ is a plane ribbon graph if and only if $A$ defines a   \pb of $G$.
\end{theorem}

\begin{figure}
\begin{center}
\begin{tabular}{ccc}
 \includegraphics[width=3.5cm]{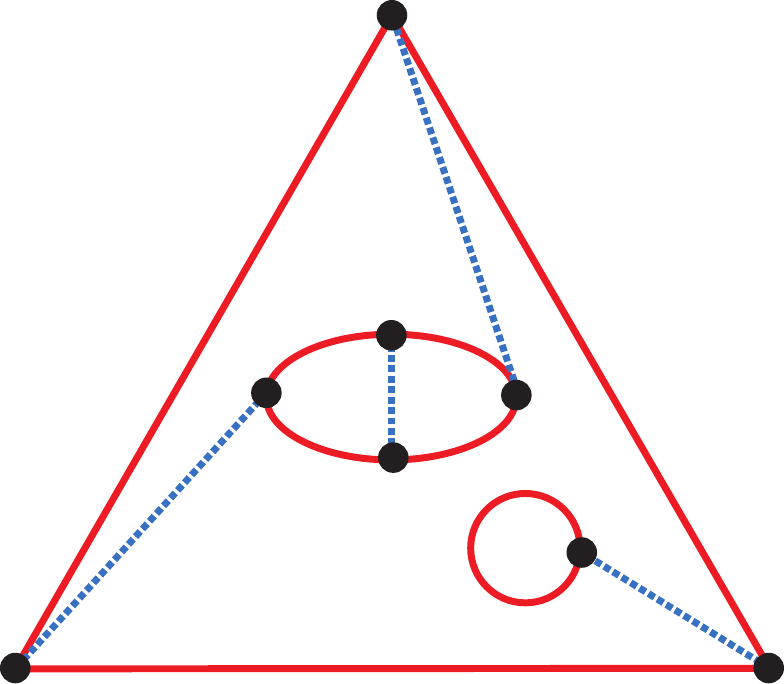} && \includegraphics[width=3.5cm]{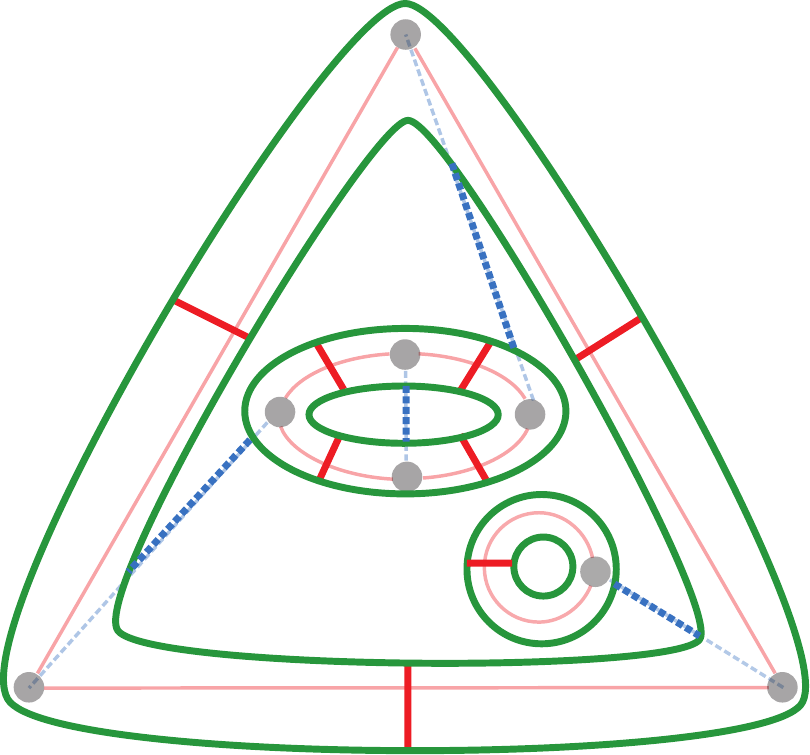} \\
  $G$. && The boundaries of $(V(G),A)$ . 
\end{tabular}

\begin{tabular}{c}
\\
\includegraphics[height=3.5cm]{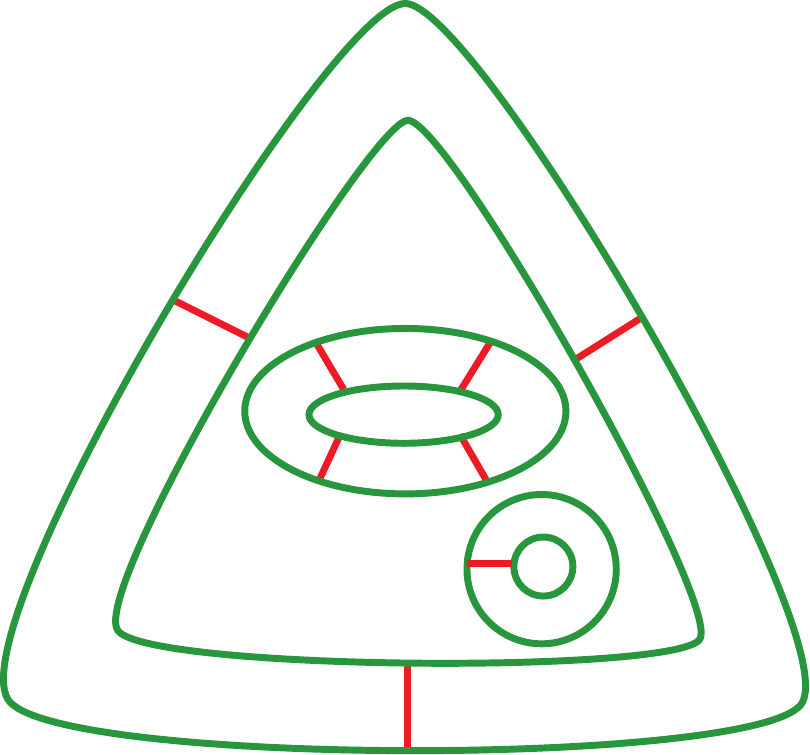} \raisebox{15mm}{\includegraphics[width=1cm]{arrow} } \includegraphics[height=3cm]{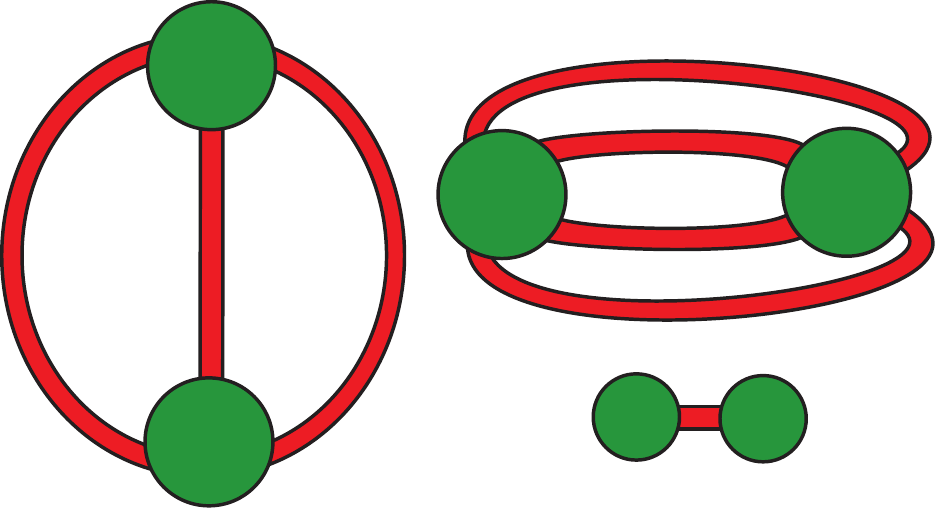}\\
  The $1$-summands arising from $A$. . 
\end{tabular}

\begin{tabular}{c}
\\
 \includegraphics[height=3.5cm]{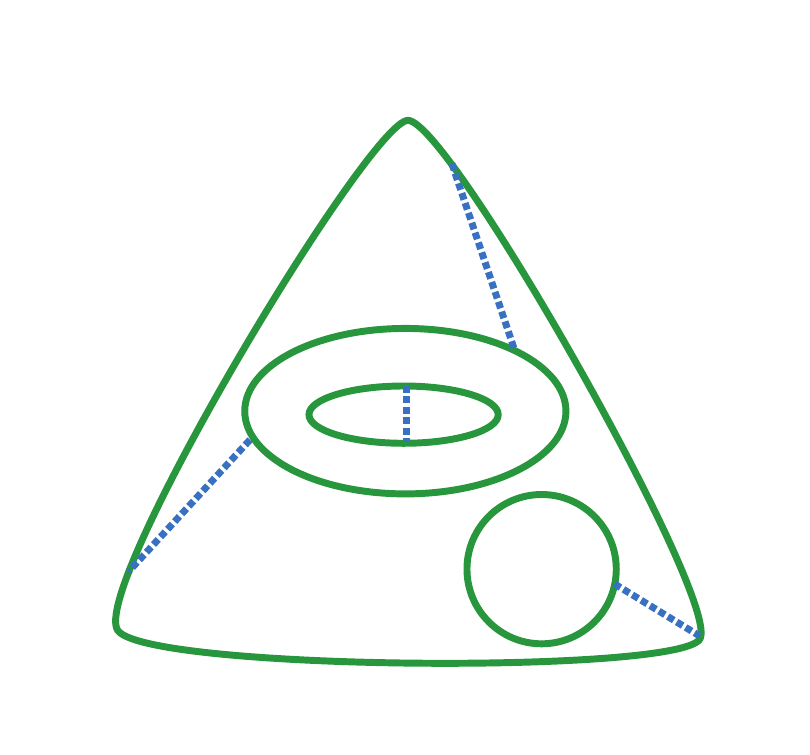} \raisebox{15mm}{\includegraphics[width=1cm]{arrow} }  \includegraphics[height=3cm]{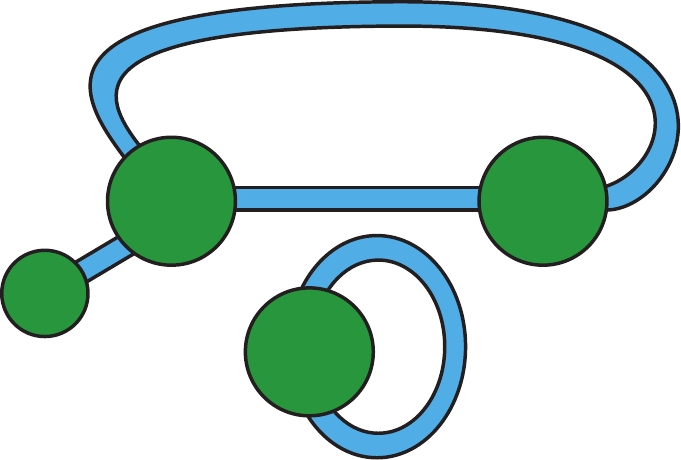} \\
   The $1$-summands arising from $E(G)\backslash A$. 
\end{tabular}
\begin{tabular}{c}
\\
  \includegraphics[width=8cm]{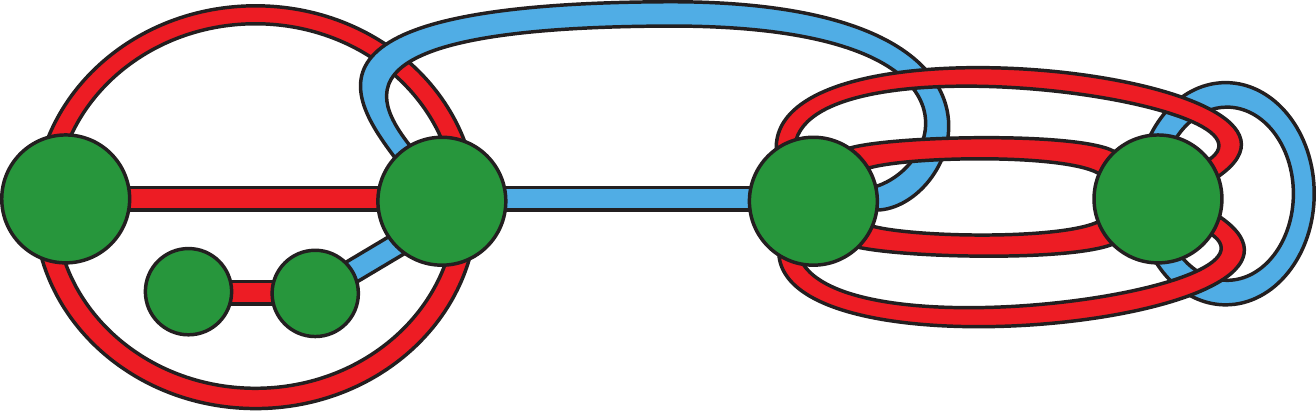}\\
  The partial dual $G^A$. 
\end{tabular}

\end{center}
\caption{An example of a construction used in the proof of Theorem~\ref{t.pchar1}.}
\label{f.c}
\end{figure}

\begin{proof} Without loss of generality, assume that $G$ is connected.

First, suppose that $G^A$ is a plane ribbon graph. Cellularly embed $G^A$ in the plane and form an arrow presentation for $(G^A)^A=G$ by taking the set of closed curves that follow the boundary components of the spanning subgraph $(V(G),A)$, and marking them with labelled arrows in the  way  described in Definition~\ref{d.pd}. This results in a set of marked non-intersecting closed plane curves, each of which corresponds to a vertex of $G$. Denote this set of marked plane curves by $\mathcal{S}$.   
Each labelled arrow in $\mathcal{S}$ touches an edge of $G$ and is labelled by that edge. Whenever a pair of arrows touches the same edge draw a line through that edge that connects  the two arrows (see Figure~\ref{f.c}). We call these lines the {\em labelling lines}.

Each closed plane curve divides the plane into a bounded  interior and an unbounded exterior region. If $c$ and $c'$ are closed plane curves, then we say that $c'$ is {\em included} in $c$ if $c'$ is in the interior of $c$, and that  $c'$ is {\em directly included} in $c$ if it is included in $c$ and there is no other closed plane curve  that is included in $c$ and includes $c'$. 
Furthermore, we say that a labelling line is  {\em directly included} in $c$ if it lies in the region of the plane bounded by $c$ and the curves that are directly included in $c$. 

Let $S_0$ be the set of marked plane curves obtained by taking all curves in $\mathcal{S}$ that meet the unbounded region of the plane and removing all arrows that are not connected by a labelling line that lies in the unbounded region.
If $S_0=\mathcal{S}$, then letting $Q_A$ be the ribbon subgraph corresponding to the arrow presentation $S_0$ and $P_A=\emptyset$ gives the required \pb of $G$. 

Otherwise, let $c_1, \ldots, c_r$ denote the marked plane curves in  $\mathcal{S}$. For each $i=1,\ldots, r$, let $S_i$  denote set of marked plane curves obtained by taking  $c_i$ and all curves directly included in it, and then removing all arrows that are not connected by a labelling line directly included in $c_i$.  

Let $\{S_{\iota_1}, \ldots , S_{\iota_s}\}$ be the set of marked curves obtained by deleting any sets from $\{S_{1}, \ldots , S_{r}\}$ that contain curves without marking arrows on them. 
Each $S_i$, for $i\in \{0, \iota_1, \ldots , \iota_r\}$,  is an arrow presentation that is equivalent to a plane ribbon subgraph of $G$.
It is clear from the construction that  $G$ can be written as sequence of $1$-sums of  $S_0,S_{\iota_1}, \ldots , S_{\iota_s}$.  Furthermore, for each $i$ and $j$,  $S_i$ and $S_j$ can either be disjoint or intersect in a single vertex; all edges of $S_i$ are either in $A$ or they are all in $E(G)\backslash A$; and  if $S_j$ shares a vertex with $S_i$, then exactly one of $S_i$ and $S_j$ will have all its edges in $A$ and the other will have all its edges in $E(G)\backslash A$. Thus, letting $P_A$ consist of all of the $S_i$ with edges in $A$, and letting $Q_A$ consist of the remaining $S_i$ gives the required \pbt.
(An example of the above argument is given in Figure~\ref{f.c}.)

Conversely, suppose that $A$ determines a \pb of $G$. 
 We will prove that $G^A$ is plane by induction on the length of a \pbt.
 If $A$ determines a \pb of length $1$ the result is trivial, and 
  if $A$ determines a \pb of length $2$ then $G^A$ is plane by Lemma~\ref{l.pchar3}.

Now suppose that the assertion holds for all ribbon graphs and edge sets that define a \pb of length less than $l$.

Suppose that $G$ is a ribbon graph, $A\subseteq E(G)$, and
$ H_1 \oplus H_2 \oplus \cdots \oplus H_l  $
is the sequence of $1$-sums in the \pb determined by $A$. Suppose also that $H_l$ and $H_i$ share a vertex.
Without loss of generality, let $E(H_l)\subseteq A$. (We may assume this  since $G^A$ is plane if and only if $G^{E(G)\backslash A}=(G^A)^*$ is, and since $A$ determines a \pb if and only if $E(G)\backslash A$ does.) Then 
\begin{multline*}
G^{E(H_l)} = (H_1 \oplus H_2 \oplus \cdots \oplus H_{l-1} \oplus H_l )^{E(H_l)} \\
  = (H_1 \oplus H_2 \oplus \cdots  \oplus (H_i \oplus H_l)  \oplus \cdots       \oplus H_{l-1} )^{E(H_l)} \\
   = (H_1 \oplus H_2 \oplus \cdots  \oplus (H_i \oplus H_l)^{E(H_l)}  \oplus \cdots       \oplus H_{l-1} ),
\end{multline*}
where the second and third equalities use the facts that  $H_l$ is the last term in the sequence of $1$-sums, and that, by Lemma~\ref{l.pchar3}, if the $1$-sum $H_i \oplus H_l$ occurs at $v$ then all other vertices of $H_i$ are also vertices of $ (H_i \oplus H_l)^{E(H_l)}$.
Also by Lemma~\ref{l.pchar3}, $(H_i \oplus H_l)^{E(H_l)}$ is plane so  the above expression gives  \pb of $G^{E(H_l)}$  of  length $l-1$. Moreover, this \pb is determined by $A\backslash  E(H_l) \subseteq E(G^{E(H_l)})$. By the inductive hypothesis, it then follows that $  (G^{E(H_l)})^{A\backslash E(H_l) }= G^A $  is plane, as required.
\end{proof}

We conclude this section by noting a few immediate corollaries of Theorem~\ref{t.pchar1}.
\begin{corollary}
An embedded graph $G$ is a partial dual of a plane graph if and only if there exists a \pb of $G$.
\end{corollary}
\begin{corollary}\label{c.pchar3} Let $k\geq 2$.
If a ribbon graph $G$ contains a $k$-connected, non-plane ribbon subgraph, then $G$ is not a partial dual of a plane graph.  
\end{corollary}
\begin{corollary}\label{c.pchar5}
Partial duals of plane graphs are planar.
\end{corollary}
\begin{corollary}\label{c.pchar4}
If a graph contains a $K_5$- or $K_{3,3}$-minor then no embedding of it   is  a partial dual of a plane graph. 
\end{corollary}

\subsection{ Relating \pbs of $G$ }
By Theorem~\ref{t.pchar1}, we see that understanding  partial duals of  plane graphs is equivalent to understanding \pbst. In this subsection we determine how all \pbs of a ribbon graph  are related to each other.

In general, a ribbon graph can admit many \pbst.
However, it turns out that any two \pbs of a ribbon graph  are related in a very simple way. Roughly speaking, the only choice that one can make in the construction of a  \pb is whether to place a  join-summand in  $P_A$ or in $Q_A$. In this subsection we will make this statement precise and relate all of the different \pbs that a ribbon graph $G$ admits.

We begin by defining  an operation on a subset of edges of a ribbon graph  that will allow us to relate all of the \pbs  that it admits.

\begin{definition}
For  $r\geq 1$, suppose that   $G=K_1 \vee K_2 \vee \cdots \vee K_r$, is a ribbon graph.  Let  $A\subseteq E(G)$ and $A'=A\Delta E(K_i)$, for some $i$. We say that $A$ and $A'$ are 
related by {\em toggling a  join-summand}. (See Figure~\ref{f.t}.)
\end{definition}

\begin{figure}


\includegraphics[height=30mm]{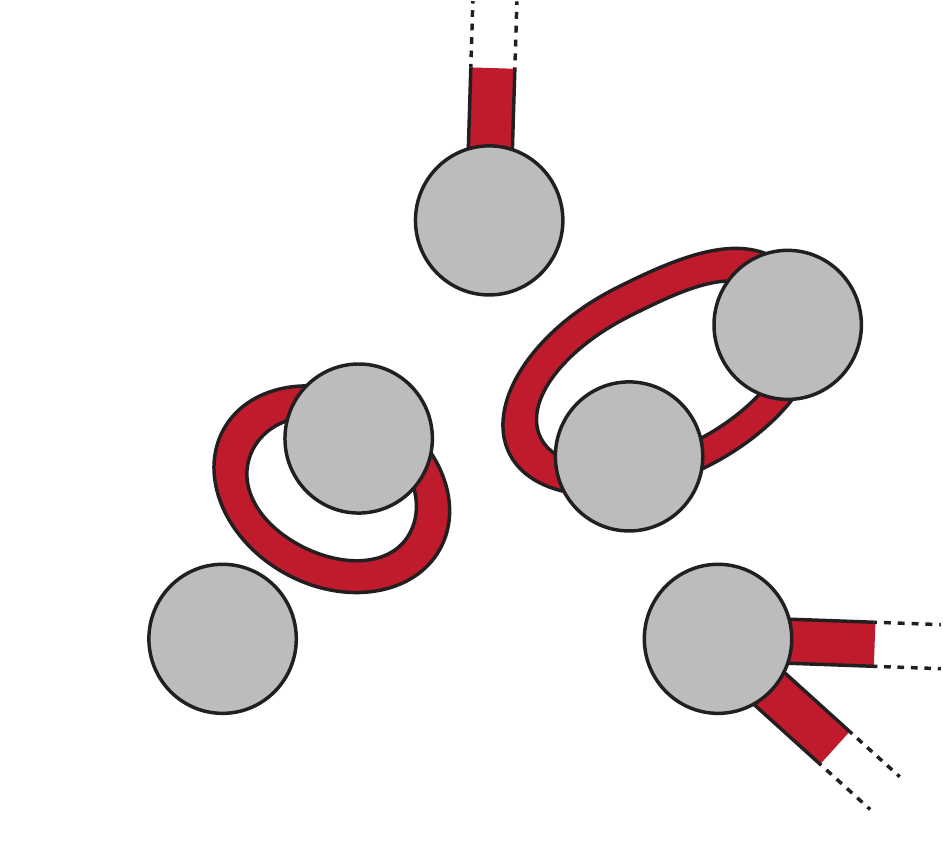}
\raisebox{12mm}{\rotatebox{180}{\includegraphics[width=1cm]{arrow}}}
\includegraphics[height=30mm]{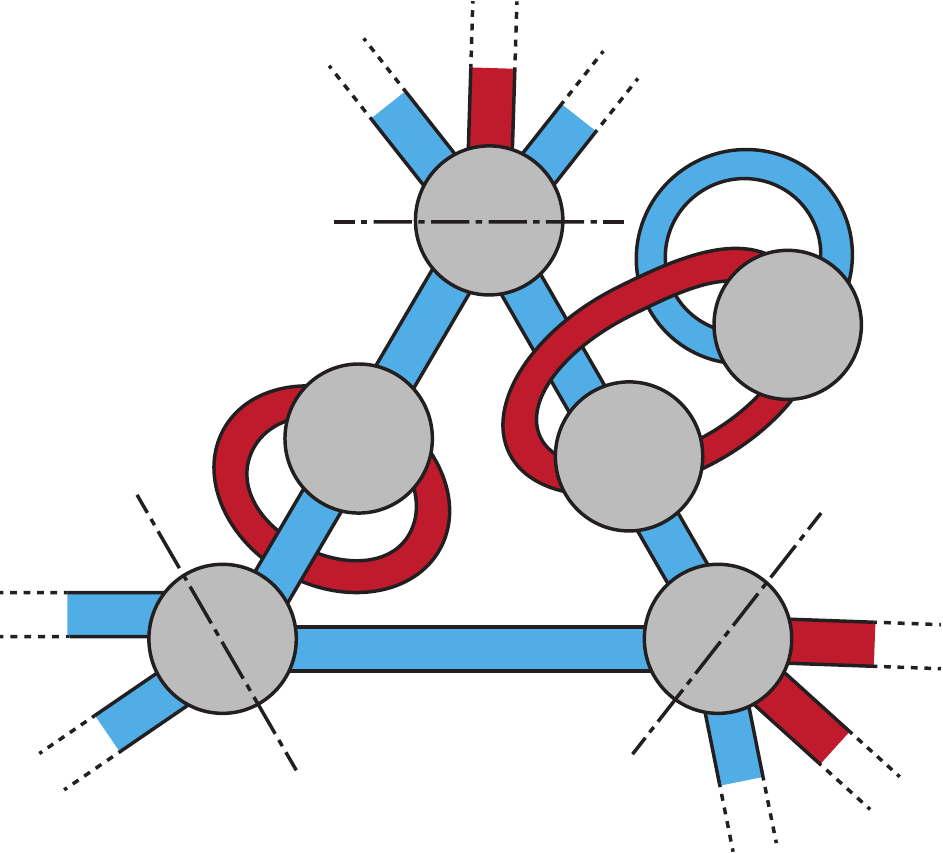}

~\hspace{4cm}
\raisebox{10mm}{\includegraphics[width=1cm]{doublearrow}}
\includegraphics[height=30mm]{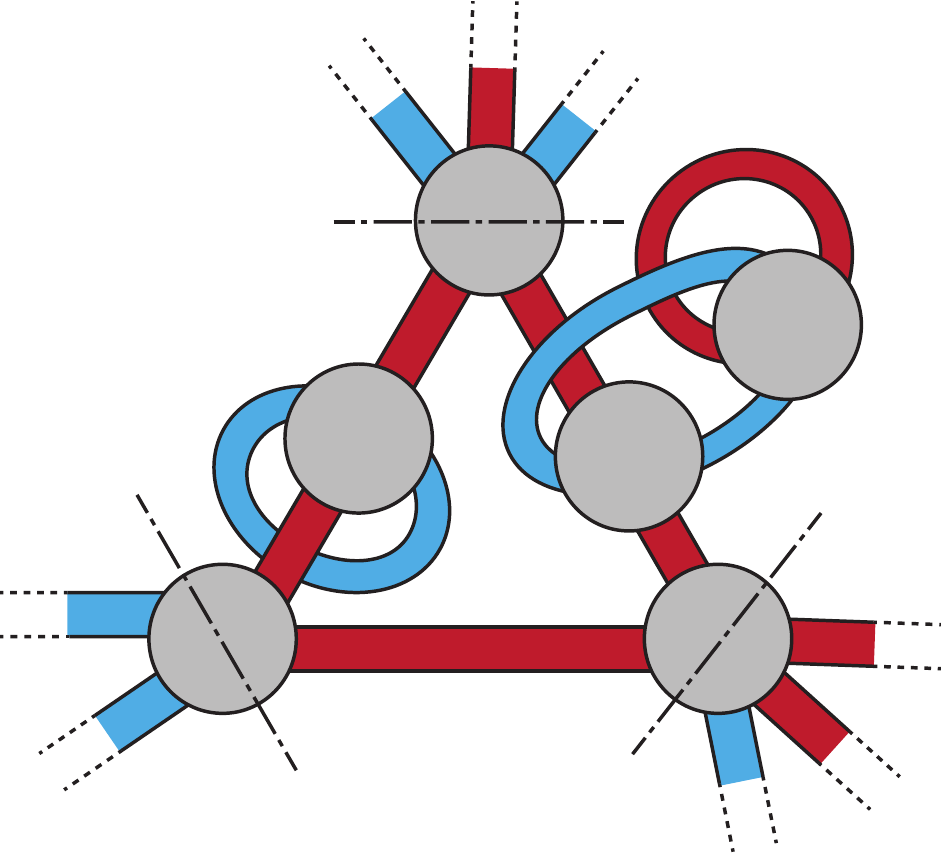}
\raisebox{10mm}{\includegraphics[width=1cm]{arrow}}
\includegraphics[height=30mm]{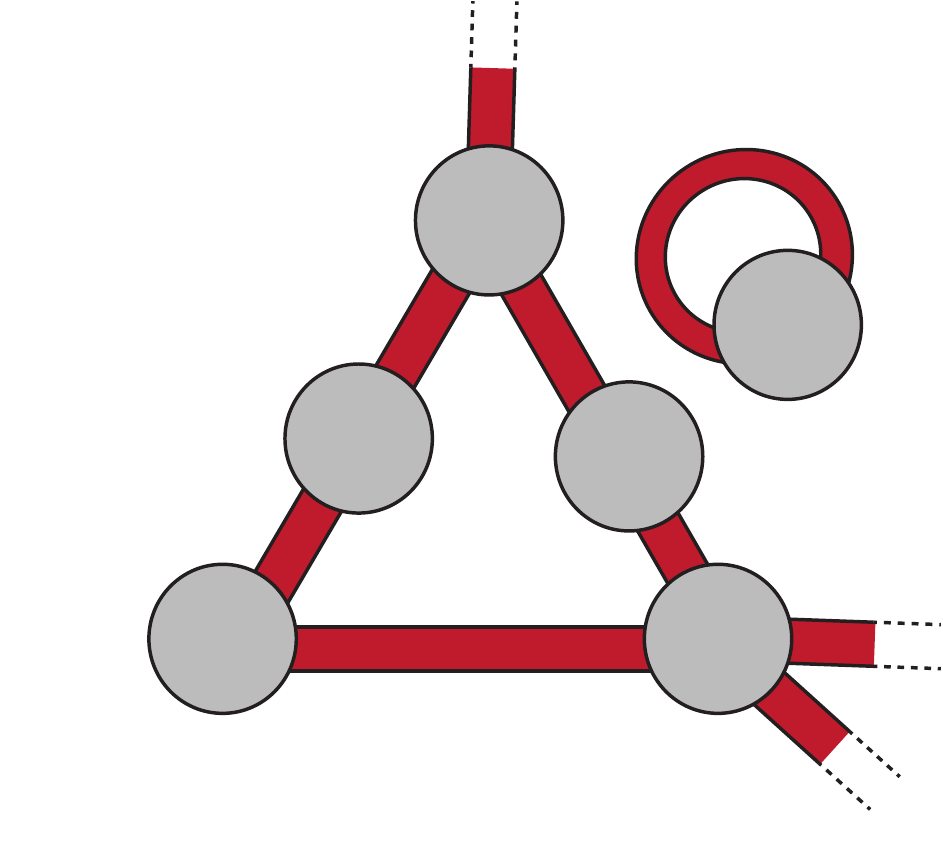}
\caption{Toggling a  join-summand, and the  ribbon subgraphs induced by $A$ and $A'$.
}
\label{f.t}
\end{figure}

The following theorem provides a simple move that relates all of the \pbs admitted by a ribbon graph.
\begin{theorem}\label{t.pbs}
Suppose that  $G$ is  a ribbon graph and that $A, B\subseteq E(G)$ both define \pbst. Then $A$ and $B$ are related by toggling join-summands. 
\end{theorem}

Example~\ref{e.deco4} gave a \pb of a ribbon graph $G$. Two other \pbs of this ribbon graph are shown in Figure~\ref{f.e7}. It is easily checked that that the edge sets defining the three \pbs are related by  toggling join-summands.

\begin{figure}
\begin{center}
\begin{tabular}{c|c}
 \includegraphics[height=6cm]{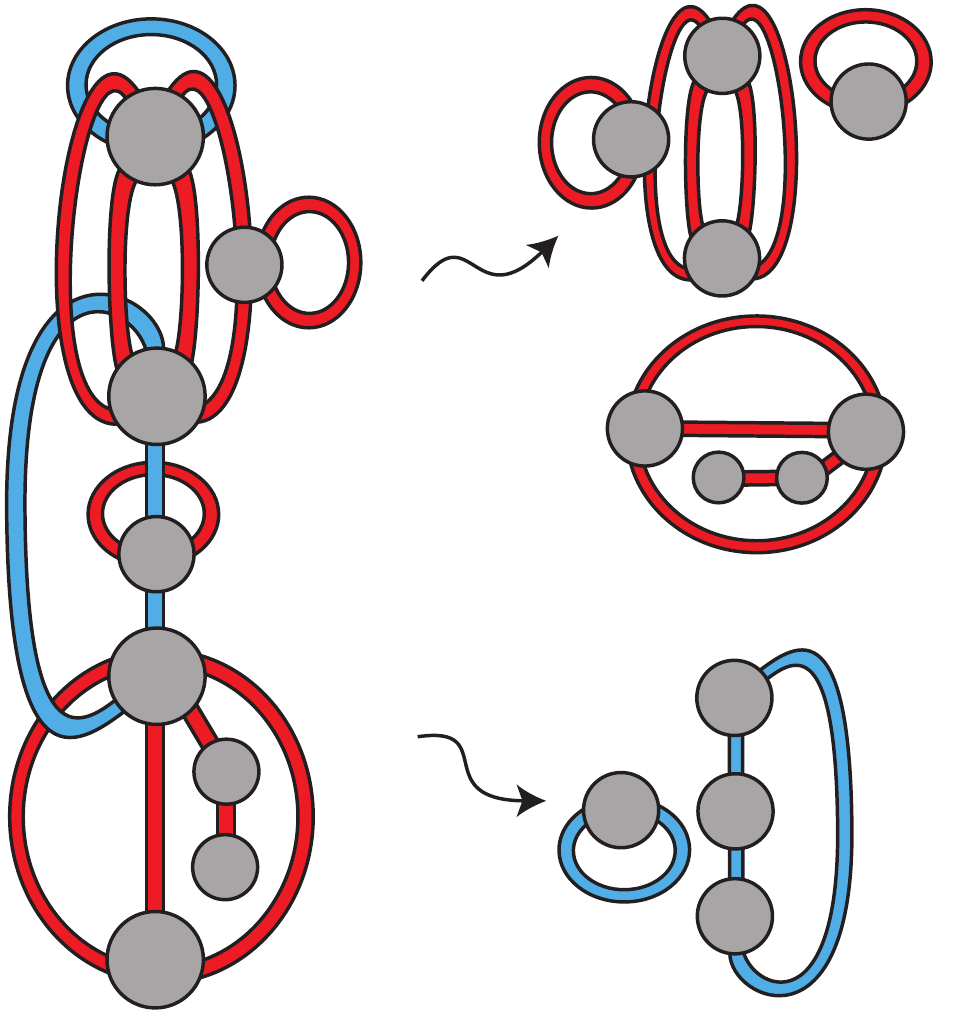}  \hspace{5mm}& \hspace{5mm}
\includegraphics[height=6cm]{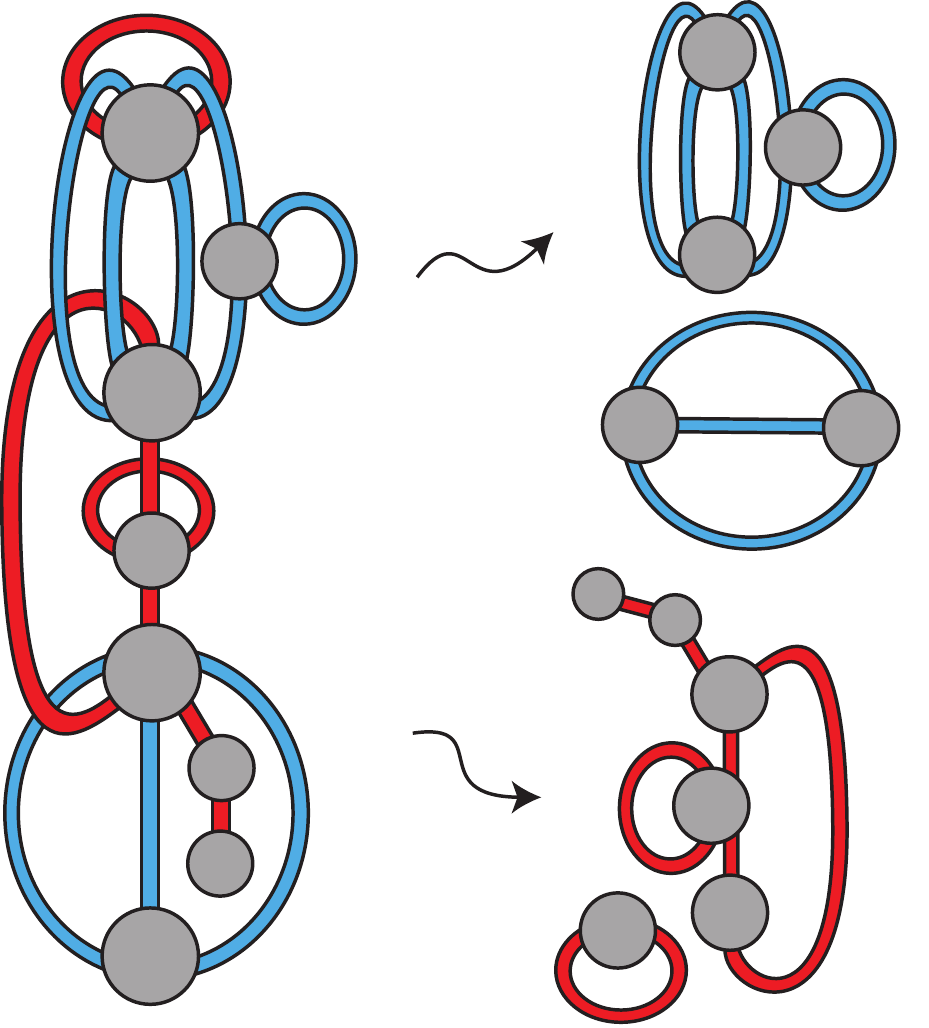}
\end{tabular}
\end{center}
\caption{Two \pbs of a ribbon graph.}
\label{f.e7}
\end{figure}

We give  some preliminary results needed to prove  Theorem~\ref{t.pbs}.

\begin{definition}
A ribbon graph is  {\em prime} if it can not be expressed as the join of two  ribbon graphs.
\end{definition}

The following result states that every ribbon graph admits a unique factorization into   prime ribbon subgraphs. 

\begin{proposition}\label{sw1}
Every ribbon graph $G$ can be written as a sequence of join-sums  
$ G= H_1 \vee \cdots \vee H_l $, 
where  each $H_i$ is prime, and $l\geq 1$. 
 Moreover, any other expression of $G$ as joins of prime join-summands can only differ from this by the order of the join-summands in the sequence.
\end{proposition}
The straight forward proof of this proposition is omitted.

We also need the following lemma.
\begin{lemma}\label{sw2}
Let $G$ be a prime, connected ribbon graph. Then either $G$ does not admit a \pbt, or it  admits exactly two. 
\end{lemma}
\begin{proof}

Suppose that $G$ admits a \pbt.  We will show that the assignment of any edge to either $A$ or $E(G)\bs A$  completely determines a \pb of $G$, and that this \pb is defined by $A$.

At each vertex $v_i$, partition the set of incident half-edges into blocks $\mathcal{A}_{i,1}, \mathcal{A}_{i,2}, \mathcal{A}_{i,3}, \ldots$ according to the following rules: place two half-edges 
 in the same block if and only if there is a path in $G$ between the two half-edges that does not pass through the vertex $v_i$.  

We will now show that the blocks at $v_i$ give rise to exactly two possible assignments of the incident edges to the sets  $A$ and $E(G)\bs A$ and that these assignments are complementary.

If there is only one block $\mathcal{A}_{i,1}$  at $v_i$ then, between  every pair of half-edges in $\mathcal{A}_{i,1}$, there is a path in $G$ that does not pass through $v_i$. It then follows that there can not be a $1$-sum occurring at $v_i$ in any \pb (as $v$ would not be  a separating vertex).  

Now suppose that the partition at $v_i$ contains more than one block and these blocks are
$\mathcal{A}_1, \cdots , \mathcal{A}_d$. Arbitrarily choose one of the cyclic orders of the half-edges incident to $v_i$.
We say that two blocks $\mathcal{A}_p$ and $\mathcal{A}_q$ {\em interlace} each other if there are half-edges $e, e'\in \mathcal{A}_p$ and $f, f'\in \mathcal{A}_q$ such that we meet the edges in the order $e,f,e',f'$ when travelling round the vertex $v_i$ with respect to the cyclic order. 

Observe that: 
\begin{itemize}
\item Every block $\mathcal{A}_p$ interlaces at least one other block $\mathcal{A}_q$.    (Otherwise $\mathcal{A}_p$ defines a join-summand and so $G$ is not prime.)
\item If $B$ is a set of blocks and $\overline{B}$ is the complementary set of blocks, then a block in $B$ interlaces a block in $\overline{B}$. (Otherwise $B$ and $\overline{B}$ define a join and so $G$ is not prime.)
\item In any \pb all of the half-edges in a block must belong to $A$ or they all belong to $E(G)\bs A$. (Since there is a path in $G$ that does not pass through $v_i$ between  every pair of half-edges in a block.)
\item In any \pb  the half-edges in interlacing  blocks must  belong to different sets $A$ or  $E(G)\bs A$. (Otherwise, since $G$ is prime, the \pb would contain a non-plane orientable $1$-summand.) 
\end{itemize}
From these observations it follows that assigning any edge incident to $v_i$ to either $A$ or to $E(G)\bs A$  determines a unique assignment of every edge that is incident to $v_i$ to either $A$ or to $E(G)\bs A$. Thus the $1$-sum at $v_i$ in the \pb is determined by the assignment of a single edge to  $A$ or $E(G)\bs A$. 

From the two cases above, and since $G$ is connected, the assignment of any edge $e$ to $A$ will  determine a unique \pbt, and the assignment of  $e$ to $E(G)\bs A$ will  determine a unique \pbt, and the result follows. 
\end{proof}

We can now prove the main result of this subsection.
\begin{proof}[Proof of Theorem~\ref{t.pbs}]
It is enough to prove the theorem for connected ribbon graphs, so assume that $G$ is connected.
Suppose that $A$ defines a \pb of $G$. 
By Proposition~\ref{sw1}, $G$ admits a unique prime factorization: $G=H_1 \vee \cdots \vee H_r$, for some $r\geq 1$.

Every \pb of $G$ is uniquely determined by choosing a \pb of $H_i$  for each $i$. Also choosing a \pb for each subgraph $H_i$ results in a \pb of $G$.
By Lemma~\ref{sw2}, each $H_i$ admits exactly two \pbs   and these are related by toggling  the edges of $H_i$  that are in $A$, and those that are not in $A $. 
Thus any  \pb can be obtained from any other one by   toggling  join-summands.
\end{proof}

We note the following corollary of Theorem~\ref{t.pbs}.
\begin{corollary}
For  $r\geq 1$, let  $G=K_1 \vee K_2 \vee \cdots \vee K_r$ be a ribbon graph. Suppose that  $A\subseteq E(G)$ defines a \pb of $G$.  Then   $A'=A\Delta E(K_i)$ also defines a \pb of $G$,  for each $i$.
\end{corollary}

\section{ Partially dual plane graphs }\label{s.join}
In this section we examine partially dual plane graphs, answering Question~\ref{q3'}.  We characterize plane partial duals 
and give a simple local move on ribbon graphs that relates them. This local move will be a key result in  the solutions to our motivating question \ref{q3}, which we answer in the next section.

The first result in this section is  the restriction of the characterization of partial duals given in Theorem~\ref{t.pchar1} to  plane graphs.  
\begin{theorem}\label{t.v1}
Let $G$ be a plane ribbon graph and $A\subseteq E(G)$. Then $G^A$ is also a plane ribbon graph if and only if $A$ defines a  \jpb of $G$.
\end{theorem}
\begin{proof}
By Theorem~\ref{t.pchar1}, $G^A$ is plane if and only if $A$ defines a  \pb of $G$. 
Since $G$ is plane, by Proposition~\ref{p.v1r}, this happens if and only if $A$ defines a   \jpb of $G$. 
\end{proof}

We will now define a move that relates all of the plane partial duals of a ribbon graph.  This move will play a key role in our applications to knot theory in Section~\ref{s.knots}. 

\begin{definition}\label{d.djs}
Let $G$ be a ribbon graph such that  $G=H_1\vee H_2$. We will say that the ribbon graph   $G^{E(H_2) } =  H_1\vee H_2^{ E(H_2)}=H_1 \vee H_2^*$ is obtained from $G$ by taking the {\em dual of  a join-summand}. The corresponding move on the set of ribbon graphs is called the  {\em dual of  a join-summand move}. Furthermore, we define an equivalence relation ``$\sim$'' on the set of ribbon graphs by setting $G\sim H$ if and only if there is a sequence of dual of  a join-summand   moves taking $G$ to $H$, or if $H=G$ or $G^*$.
\end{definition}

The following theorem shows that the dual of  a join-summand move provides a way to determine all of the plane partial duals of a plane graph. 
\begin{theorem}\label{t.sim}
Let $G$ and $H$ be plane graphs. Then $G$ and $H$ are partial duals if and only if $G\sim H$.
\end{theorem}

To prove the theorem, we need  the following lemma.
\begin{lemma}\label{sum6}
Let $G$ be a ribbon graph, $G=H_1\vee \cdots \vee H_k$, and  $A=\bigcup_{i=l}^{k}  E(H_i)$. Then $G\sim G^A$. 
\end{lemma}
\begin{proof}
Let  $K_1, \ldots , K_p$ be the non-trivial connected components of $G-E\left(H_1\vee \cdots \vee H_{l-1} \right) $. Then $G$ can  be written as  
$ H_1\vee \cdots \vee H_{l-1} \vee K_1 \vee \cdots \vee K_p$, 
where no joins involve two vertices of the  $K_i$'s.   It follows that
\[G^A =  \left( H_1\vee \cdots \vee H_{l-1} \vee K_1 \vee \cdots \vee K_p  \right)^A = 
H_1\vee \cdots \vee H_{l-1} \vee K_1^* \vee \cdots \vee K_p^* , \]
where the second equality follows by Lemma~\ref{sum5}.
From this identity it follows that $G^{A} $ can be obtained by dual of  a join-summand moves which act on $K_1, \ldots , K_p$, and so $G\sim G^A$.
\end{proof}

\begin{proof}[Proof of Theorem~\ref{t.sim}.]
Sufficiency is obvious. To prove necessity, assume that $H=G^A$. Since both $G$ and $G^A$ are plane, by Theorem~\ref{t.v1}, $A$ defines a  \jpb of $G$. If $A=\emptyset$ or $E(G)$ the result is trivial, so assume that this is not the case. 
We then have that $G= H_1\vee \cdots \vee H_l $, for $l\geq 2$. 

Let $A_i  =   A\cap E(H_i) $.   In the sequence $(A_1, \ldots , A_l)$, suppose the first  non-empty set   occurs at position $\iota_1$  and, as we read along the sequence, toggles between being non-empty and empty  at positions $\iota_2$,  \ldots  ,$\iota_r$ in this order.
Letting $B_i=\cup_{j=i}^l A_j$, then, using  Proposition~\ref{p.pd2}, we can write 
\[G^A= (H_1\vee \cdots \vee H_l )^A =  ( \cdots ((H_1\vee \cdots \vee H_l )^{B_{\iota_1}})^{B_{\iota_2}} \,\cdots  )^{B_{\iota_r}},
\]
and it follows from Lemma~\ref{sum6} that $G\sim G^A$.
\end{proof}

\section{Relating the ribbon graphs of a link diagram}\label{s.knots}
In this final section we apply the graph theoretical results of Sections~\ref{s.char} and~\ref{s.join} to answer Questions~\ref{q2} and~\ref{q3}.

\subsection{Link diagrams with the same graph   }

Although there is a unique link diagram associated with a Tait graph, in general, a ribbon graph can have many plane partial duals, and therefore $ \calD(G)$ can contain more than one distinct link diagram, i.e. many distinct link diagrams can give rise to the same set of ribbon graphs.
In this subsection we determine how these link diagrams are related to each other. 
We will see that the summand flip move, defined below, relates all link diagrams that  are presented by the same ribbon graphs.

Let $D_1\subset S^2$ and $D_2\subset S^2$ be link diagrams. The {\em connected sum} $D_1\# D_2\subset S^2$ of   $D_1$ and $D_2$ is the link diagram formed by, for $i=1,2$, choosing a disc $\mathfrak{D}_i$ on $S^2$ that intersects $D_i$ in an arc $\alpha_i$, deleting the interior of each $\mathfrak{D}_i$, and identifying the boundaries of $S^2\bs  \mathfrak{D}_1$ and   $S^2\bs  \mathfrak{D}_2$ in such a way that each end point of $\alpha_1$ is identified with a distinct endpoint of $\alpha_2$. This process is illustrated in the following figure.

\begin{center}
\begin{tabular}{ccc}
\includegraphics[height=1.5cm]{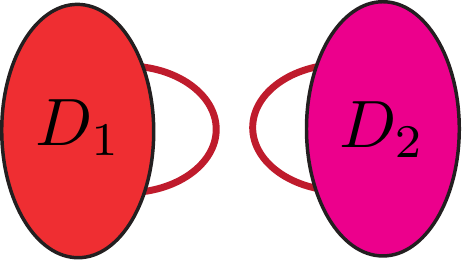} & \raisebox{7mm}{\includegraphics[width=1cm]{arrow} }  & \includegraphics[height=1.5cm]{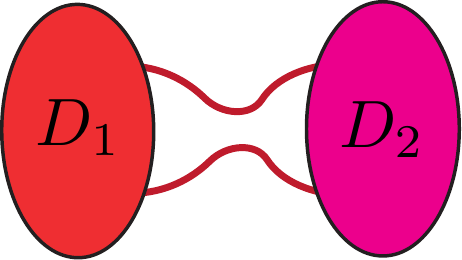} \\
 $D_1$ and $D_2$ && $D_1\# D_2$
\end{tabular}
\end{center}

For reference later, we record the following two  well-known (and obvious) results.
\begin{lemma}\label{sums}
Let $G$ be a plane graph such that $G=H_1\vee H_2$, then $D(H_1\vee H_2)=  D(H_1)\# D(H_2)$. Moreover, there is a natural correspondence between the  arc on $H_1\vee H_2$ that defines the join,  and the arcs on $D(H_1)$ and $D(H_2)$ used to form the connected sum.
\end{lemma}

\begin{lemma}\label{l.diaofdual}
Let $G$  be a signed plane ribbon graph. Then 
$D(G) = D(G^* )$. 
\end{lemma}

\begin{definition}\label{d.lflip}
Let $D_1$ and $D_2$ be link diagrams. Then we say that $D_1$ and $D_2$ are related by a  {\em summand flip} if and only if $D_2$ can be obtained from $D_1$ by the following process: orient $S^2$ and choose an disc $\mathfrak{D}$ in $S^2$ whose boundary intersects $D_1$ transversally in exactly two points $a$ and $b$. Cut out   $\mathfrak{D}$ and glue it back in  such a way that the orientations of $\mathfrak{D}$ and $S^2\bs \mathfrak{D}$ disagree and the points $a$ on the boundaries of $\mathfrak{D}$ and $S^2\bs  \mathfrak{D}$ are identified, and  the points $b$ on the boundaries of $\mathfrak{D}$ and $S^2\bs \mathfrak{D}$ are identified. See Figure~\ref{f.sf}.

We will say that two  link diagrams $D_1$ and $D_2$   are {\em related by summand-flips}, written $D_1\sim D_2$, if and only if there is a sequence of summand-flips taking $D_1$ to $D_2$.
\end{definition}

\begin{figure}
\begin{center}
\begin{tabular}{ccccc}
\includegraphics[height=2cm]{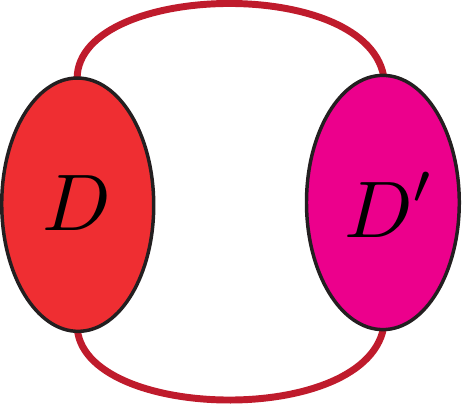} & \raisebox{7mm}{\includegraphics[width=1cm]{arrow} }  & \includegraphics[height=2cm]{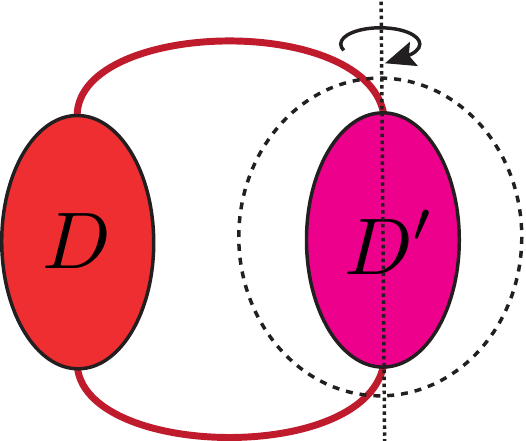}& \raisebox{7mm}{\includegraphics[width=1cm]{arrow} }  & \includegraphics[height=2cm]{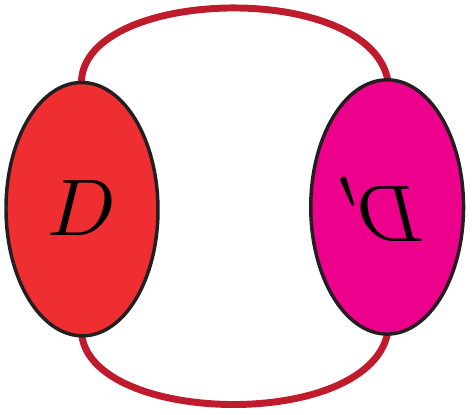} \\
 $D_1 = D\#D'$  &&  cut, flip and glue   && $ D_2$
\end{tabular}
\end{center}
\caption{A summand-flip.}
\label{f.sf}
\end{figure}

The following theorem is the main result of this section. It describes how link diagrams arising from the same ribbon graph are related. 
\begin{theorem}\label{t.diagrams}
Let $G$ be a signed ribbon graph and $D\in \mathcal{D}(G)$. Then $D'\in \mathcal{D}(G)$ if and only if $D\sim D'$.
\end{theorem}
The proof of this result appears at the end of this subsection. It will follow from Lemma~\ref{t.sg1} below. Before giving the proof of Theorem~\ref{t.diagrams}, we note  a straight-forward, but  important, corollary. 
\begin{corollary}
Let $D,D'\in \mathcal{D}(G)$. Then the links corresponding to $D$ and $D'$ are isotopic.  
\end{corollary}

Theorem~\ref{t.diagrams} will follow from  Lemma~\ref{t.sg1} which relates the dual of a join-summand move on a ribbon graph (see Definition~\ref{d.djs}) to the summand-flip move on a link diagram.
\begin{lemma}\label{t.sg1}
 Let $G$ and $G'$ be signed plane graphs. Then   $D(G)\sim D(G')$ if and only if $G\sim G'$.
\end{lemma}
\begin{proof}
It is enough to show that $G$ and $G'$ are related by a single dual of a join-summand move if and only if $D(G)$ and $ D(G')$ are related by a single summand-flip.  
Suppose that  $G$ and $G'$ are related by a single dual of a join-summand move. Then we can write $G=H_1\vee H_2$ and $ G'=H_1\vee H_2^* = (H_1\vee H_2)^{E(H_2)}$. 
If $H_1\cap H_2=v$, then 
we can form the partial dual $ (H_1\vee H_2)^{E(H_2)}$ in the following way: \begin{enumerate}
\item detach $E(H_1)\cup (V(H_1)\backslash v)$ from the vertex $v$ by `cutting' along the  intersections of $H_1$ and $v$. Record the position of the detached edges by using labelled marking arrows in such a way that  $H_1\vee H_2$ can be recovered by identifying  the arrows  of the same label on $H_2$ and $E(H_1)\cup (V(H_1)\backslash v)$. 
\item Form the dual $H_2^*$ of $H_2$, retaining the labelled arrows on the boundary. 
\item Attach $E(H_1)\cup (V(H_1)\backslash v)$ to $H_2^*$ by identifying the arrows of the same label. The resulting ribbon graph is $(H_1\vee H_2)^{E(H_2)}$.
\end{enumerate}
This process is illustrated in   Figure~\ref{sg1f1}.

\begin{figure}
\begin{center}
\hspace{-3cm}\includegraphics[height=20mm]{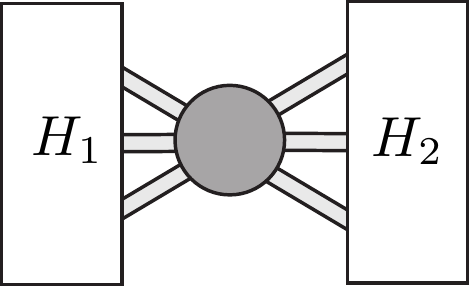} \raisebox{9mm}{\includegraphics[width=1cm]{arrow} } 
\includegraphics[height=20mm]{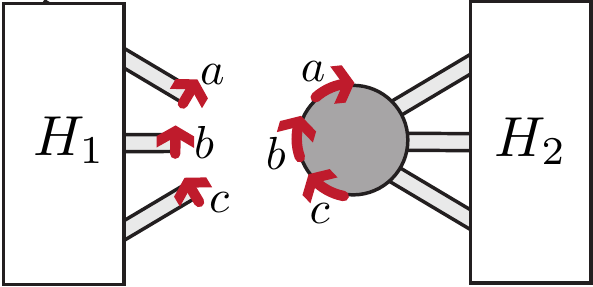} 

\vspace{5mm}

\hspace{3cm}\raisebox{9mm}{\includegraphics[width=1cm]{arrow} } 
\includegraphics[height=20mm]{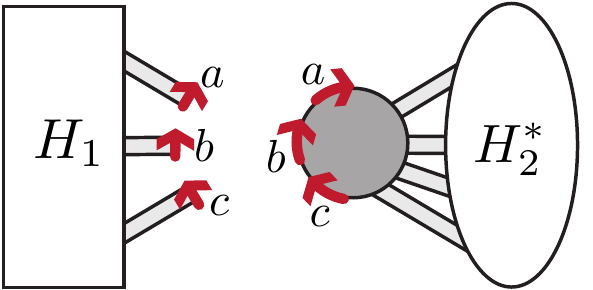}  \raisebox{9mm}{\includegraphics[width=1cm]{arrow} } 
\includegraphics[height=20mm]{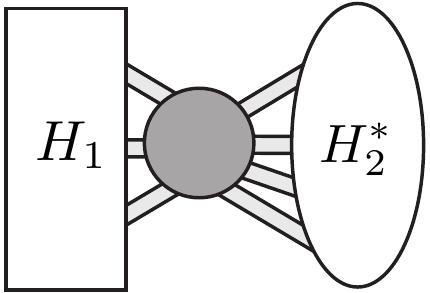}
\end{center}
\caption{Forming $(H_1\vee H_2)^{E(H_2)}$ as  in the proof of Lemma~\ref{t.sg1}.}
\label{sg1f1}
\end{figure}

  $G=H_1\vee H_2$ is a plane graph so it admits a unique embedding in the sphere $S^2$. We consider how the above construction of the partial dual $(H_1\vee H_2)^{E(H_2)}$ acts on the embedded ribbon graph $ H_1\vee H_2 \subset S^2$. Start off with an embedding of $H_1\vee H_2$ in $S^2$ where the vertex $v$ forms the southern hemisphere. Detaching $E(H_1)\cup (V(H_1)\backslash v)$ gives an embedding of $H_2$ in $S^2$ with marking arrows recording how $H_1$ was attached. Since $v\subset H_2\subset S^2$ is the southern hemisphere, the dual $H^*\subset S^2$ is embedded in the northern hemisphere with the marking arrows sitting on the equator. Finally, to attach and embed  $H_1$ (in order to obtain an embedding of   $H_1\vee H_2^* =(H_1\vee H_2)^{E(H_2)}$ in $S^2$) we ``flip over'' $H_1$ and embed it in the  southern hemisphere using the identifying arrows. The resulting embedding of $(H_1\vee H_2)^{E(H_2)}$ is obtained from the embedding of $H_1\vee H_2$ by a single dual of  a join-summand move acting on $H_1$.
The argument is illustrated in   Figure~\ref{sg1f2}.

\begin{figure}
\begin{center}
\hspace{-2cm}\includegraphics[height=30mm]{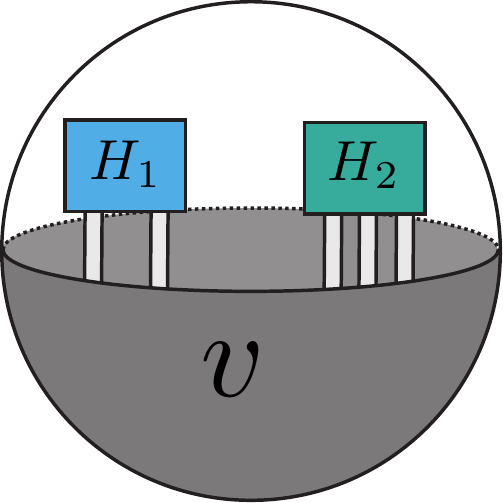}  \raisebox{8mm}{\includegraphics[width=1.5cm]{arrow} }
\includegraphics[height=30mm]{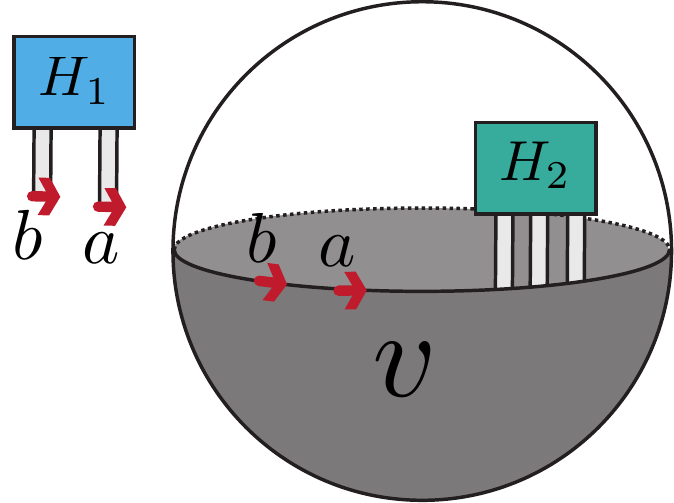}  

\vspace{5mm}

\hspace{2cm}\raisebox{8mm}{\includegraphics[width=1.5cm]{arrow} }
\includegraphics[height=30mm]{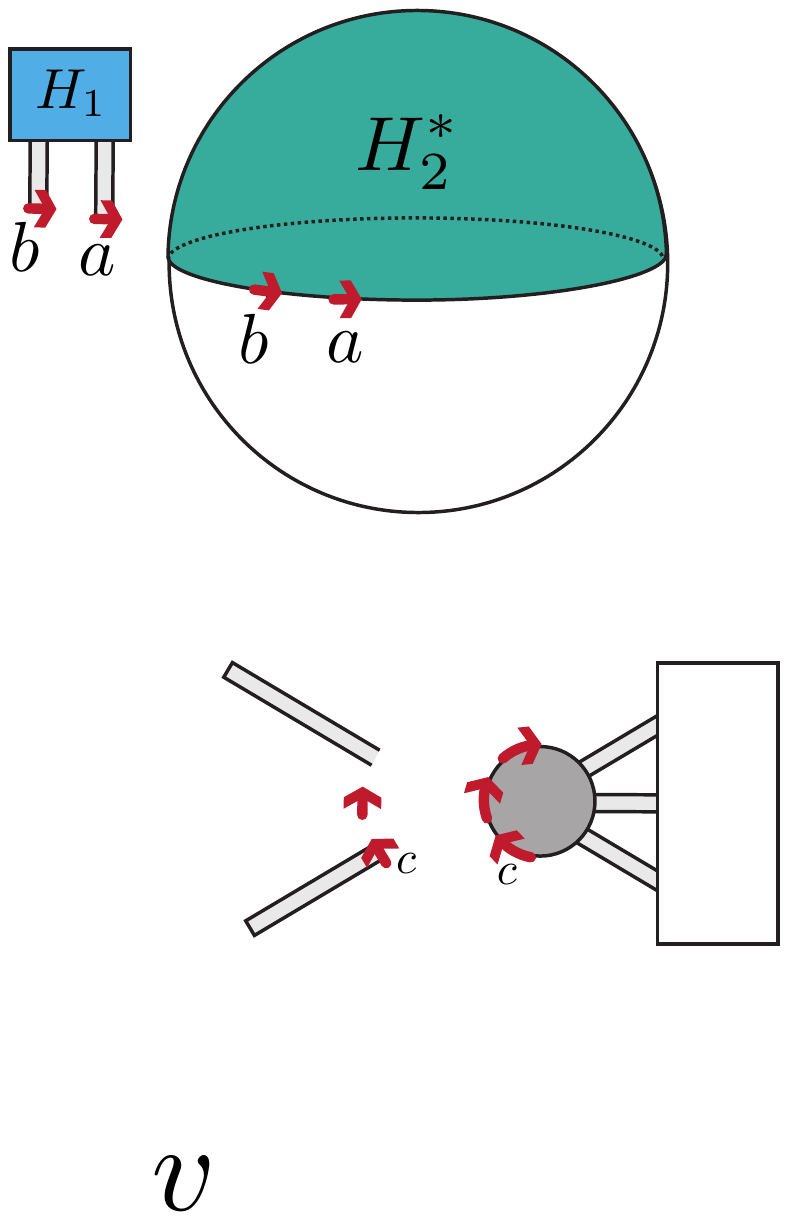}  \raisebox{8mm}{\includegraphics[width=1.5cm]{arrow} }
\includegraphics[height=30mm]{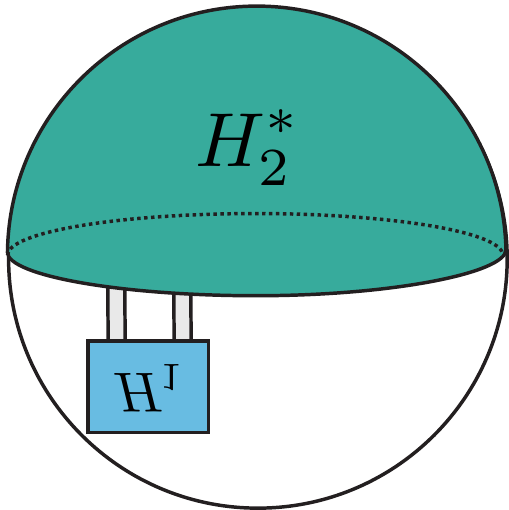}
\end{center}
\caption{Embedding $(H_1\vee H_2)^{E(H_2)}$ as  in the proof of Lemma~\ref{t.sg1}.}
\label{sg1f2}
\end{figure}

Since the embeddings of $G=H_1\vee H_2$ and $ G'=H_1\vee H_2^* $  in $S^2$ are unique, it follows that the embeddings of $G$ and $G'$ are related as in Figure~\ref{sg1f2}. The lemma then follows by considering the corresponding link diagrams $D(G)$ and $D(G')$ as illustrated in Figure~\ref{f.sg1}.
Since the link diagrams $D(H_2^*)$ and  $D(H_2)$ are equal by Lemma~\ref{l.diaofdual}, it follows that  
$G$ and $G'$ are related by a single dual of  a join-summand move if and only if $D(G)$ and $ D(G')$ are related by a single summand-flip. This completes the proof of the lemma.
\end{proof}

\begin{figure}
\begin{center}
\begin{tabular}{ccc}
  \includegraphics[height=30mm]{ld5}  &
\raisebox{10mm}{  \begin{tabular}{c} \includegraphics[width=3cm]{doublearrow}   \\  dual of  a join-summand  \end{tabular}}
&
\includegraphics[height=30mm]{ld8} 
\\   
\raisebox{3mm}{\rotatebox{90}{  \small{$D(H_1\vee H_2)$}}  } \includegraphics[height=30mm]{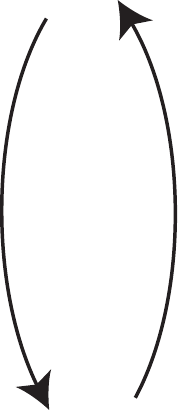}  \rotatebox{90}{  \small{$\T(D(H_1)\# D(H_2))$}  }&
&
\raisebox{3mm}{\rotatebox{90}{  \small{$D(H_1\vee H_2^*)$}}  } \includegraphics[height=30mm]{arrowtwo}  \rotatebox{90}{  \small{$\T(D(H_1)\# D(H_2^*))$}  }
 \\
 \includegraphics[height=30mm]{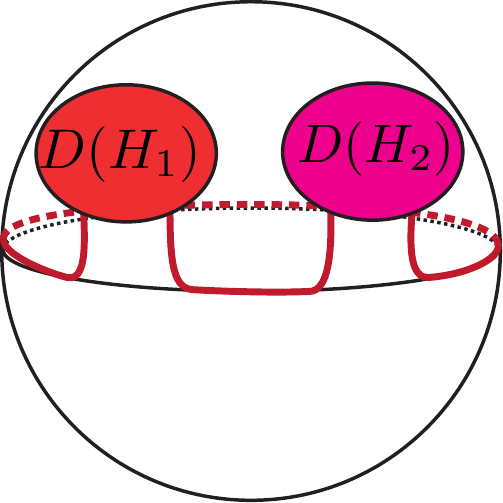}  &
\raisebox{10mm}{  \begin{tabular}{c} \includegraphics[width=3cm]{doublearrow}   \\  flipping a summand  \end{tabular}}
&
\includegraphics[height=30mm]{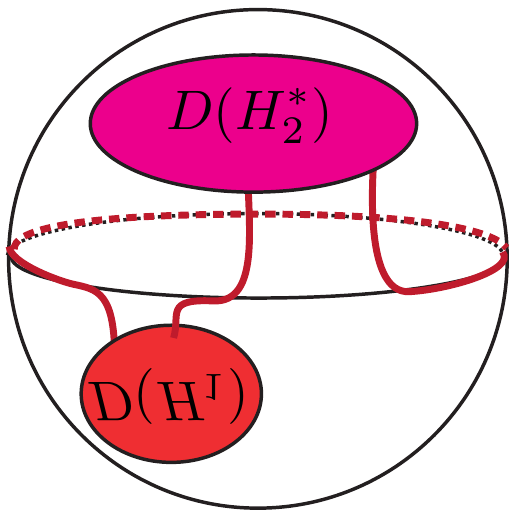} .

\end{tabular}
\end{center}
\caption{A figure used in the proof of Lemma~\ref{t.sg1}.}
\label{f.sg1}
\end{figure}

\begin{remark}
Lemma~\ref{t.sg1} {\em does not} say that the medial ribbon graphs of $G$ and   $G'$ are equal. (See \cite{EMM} for details on the relations between medial graphs and partial duals.) To see why this is, notice that in the proof of Lemma~\ref{t.sg1} the interior and the exterior of the vertex $v$, where the join occurs, are ``switched''. This switching may add half-twists to the edges. Consequently, twists may be added to the medial graphs. For link diagrams this switching is not a problem since links  are one dimensional.  In particular, this means that Lemma~\ref{t.sg1}
 does not contradict the well-known fact that the medial graphs of $G$ and $H$ are equal if and only if $H=G$ or $H=G^*$.
\end{remark}

\begin{proof}[Proof of Theorem~\ref{t.diagrams}.]
By Equation~\eqref{e.cald}, $D,D'\in \mathcal{D}(G)$ if and only if $D=D(H)$ and $D'=D(H^A)$, where $H$ is a partial dual of $G$, and $H$ and $H^A$ are both plane graphs.  By Theorem~\ref{t.sim},  this happens if and only if $H\sim H^A$. 
Finally, by Lemma~\ref{t.sg1},  $H\sim H^A$ if and only if $D=D(H)\sim D(H^A)=D'$.
\end{proof}

\subsection{Answering our motivating questions  }
Recall our motivating questions:
\begin{enumerate}
\renewcommand{\theenumi}{Q\arabic{enumi}}
\item How are the ribbon graphs of a link diagram related?
\item   Which ribbon graphs arise as ribbon graphs of a link diagram?
\item  What is the relation between link diagrams that are presented by  the same signed ribbon graph?  
\end{enumerate}
These questions ask for the extensions of the  well-known, basic properties of Tait graphs, \eqref{t1}-\eqref{t3}, to the ribbon graphs of a link digram.

We will now summarize the solutions to \eqref{q1}-\eqref{q3}.

\noindent \underline{Question \eqref{q1}:}
This question was answered in \cite{Mo2} and by Chmutov  in \cite{Ch1}, and comes from Proposition~\ref{p.pd3}: the ribbon graphs of a link diagram are all partial duals of each other.

\noindent \underline{Question \eqref{q2}:}
Corollary~\ref{c.pd6} provides a graph theoretical formulation of Question~\ref{q2}:
which ribbon graphs are the partial duals of plane graphs?  By applying Theorem~\ref{t.pchar1} we get the following answer.
\begin{theorem}\label{t.pchar2}
A  (signed) ribbon graph $G$ presents a link diagram if and only if it admits a \pbt.
\end{theorem}
 This result can also be used to deduce other structural results about the set of ribbon graphs that present links.
\begin{corollary}Let $k\geq 2$.
If a ribbon graph $G$ contains a $k$-connected, non-plane ribbon subgraph, then $G$ is not the ribbon graph of a link diagram.  
\end{corollary}
\begin{corollary}
If a graph contains a $K_5$- or $K_{3,3}$-minor then no embedding of it  presents a link diagram.
\end{corollary}

\noindent \underline{Question \eqref{q3}:} 
This question was answered by Theorem~\ref{t.diagrams}, where it was shown that all diagrams that are presented by the same signed ribbon graphs are related in a very simple way: they are summand flips of one another.   In particular, this means that (the isotopy class of) every link in $3$-space will define a unique set of ribbon graphs and so link invariants may be defined on these sets.   

\medskip

We conclude with a few remarks on the graph theory that this study has introduced. The connection between genus and separability of ribbon graphs that is opened up by Theorem~\ref{t.pchar1} is one of interest beyond the applications to knot theory given here. In \cite{Mo6}, a more general study of the connections between separability and the genus of a partial dual is given. In particular, it is shown that Theorems~\ref{t.pchar1}, \ref{t.pbs}, and~\ref{t.sim} can be extended to characterize partial duals of graphs in the real projective plane. These characterizations of the genus of a partial dual in terms of separability, do not extend beyond  graphs in the plane and real projective plane. For higher genus graphs, a concept that uses higher connectivity is required, and the characterization of the partial duals of higher genus graphs is a work in progress.

\section*{Acknowledgements}
I would like to thank Scott Carter, Sergei Chmutov, and Martin Loebl for helpful conversations.

\end{document}